\documentclass[11pt,a4paper]{article}

\usepackage[english]{babel}
\usepackage[utf8]{inputenc} 
\usepackage[T1]{fontenc} 
\usepackage{amsmath,amsfonts,amssymb,amsthm}
\usepackage{textcomp}
\usepackage{stmaryrd}
\usepackage[a4paper]{geometry}
\usepackage{dsfont}
\usepackage{pifont}
\usepackage{hyperref}
\usepackage{tikz}

\newcommand{\N}{\mathbb{N}}
\newcommand{\Z}{\mathbb{Z}}

\newcommand{\R}{\mathbb{R}}
\newcommand{\C}{\mathbb{C}}

\newcommand{\dps}{\displaystyle}

\newcommand{\mm}{\mathcal{M}}

\newcommand{\gs}{\mathfrak{S}}

\newcommand{\Lb}{\mathcal{L}}

\newcommand{\fl}{\mathcal{F}\!\ell}
\newcommand{\hooklongrightarrow}{\lhook\joinrel\longrightarrow}
\newcommand{\circled}[1]{\raisebox{.3pt}{\textcircled{\raisebox{-1pt} {#1}}}}

\DeclareMathOperator{\im}{Im}

\DeclareMathOperator{\vect}{Vect}

\DeclareMathOperator{\gl}{GL}

\DeclareMathOperator{\h}{H}

\newenvironment{changemargin}[2]{%
\begin{list}{}{%
\setlength{\topsep}{0pt}%
\setlength{\leftmargin}{#1}%
\setlength{\rightmargin}{#2}%
\setlength{\listparindent}{\parindent}%
\setlength{\itemindent}{\parindent}%
\setlength{\parsep}{\parskip}%
}%
\item[]}{\end{list}}

\newtheorem{theo}{Theorem}[section]
\newtheorem{prop}[theo]{Proposition}
\newtheorem{lemma}[theo]{Lemma}
\newtheorem{cor}[theo]{Corollary}
\theoremstyle{definition}
\newtheorem{de}[theo]{Definition}
\newtheorem{rmk}[theo]{Remark}

\title{Production of faces of the Kronecker cone containing only stable triples}
\author{Maxime Pelletier\thanks{Univ Lyon, Université Claude Bernard Lyon 1, CNRS UMR 5208, Institut Camille Jordan, 43 blvd. du 11 novembre 1918, F-69622 Villeurbanne cedex, France (\texttt{pelletier@math.univ-lyon1.fr})}}
\date{\today}

\begin{document}
\maketitle

\begin{abstract}
One way to study the Kronecker coefficients is to focus on the Kronecker cone, which is generated by the triples of partitions corresponding to non-zero Kronecker coefficients. In this article we are interested in producing particular faces of this cone, formed of stable triples (a notion defined by J. Stembridge in 2014), using many geometric notions -- principally those of dominant and well-covering pairs -- and results of N. Ressayre. This extends a result obtained independently by L. Manivel and E. Vallejo in 2014 or 2015, expressed in terms of additive matrix. To illustrate the fact that it allows to produce quite a few new faces of the Kronecker cone, we give at the end of the article details about what our results yield for ``small dimensions''.
\end{abstract}

\section{Introduction}

The Kronecker coefficients are defined as the multiplicities arising in the decomposition of the tensor product of irreducible representations of symmetric groups. More precisely, if $k$ is a positive integer, denote by $\gs_k$ the symmetric group of permutations of $\llbracket 1,k\rrbracket$. Then the irreducible complex representations of $\gs_k$ are indexed by the partitions of the integer $k$ (i.e. the non-increasing finite sequences of positive integers -- that we will call parts -- whose sum is $k$) and, for any such partition $\alpha$, we denote by $M_\alpha$ the corresponding irreducible complex $\gs_k$-module. Then, if $\alpha$ and $\beta$ are partitions of $k$,
\[ M_\alpha\otimes M_\beta=\bigoplus_{\gamma\vdash k}M_\gamma^{\oplus g_{\alpha,\beta,\gamma}}, \]
and the coefficients $g_{\alpha,\beta,\gamma}$ are the Kronecker coefficients. For simplicity of notations we can extend the definition of these coefficients to triples of partitions of different integers by setting $g_{\alpha,\beta,\gamma}=0$ in that case. One axis of research concerning these Kronecker coefficients is to be interested in studying whether their value is 0 or not. As a consequence we fix two positive integers $n_1$ and $n_2$ and, denoting for any partition $\alpha$ its length -- i.e. the number of its parts -- by $\ell(\alpha)$, we are interested in the following set:

\begin{de}
The set
\[ \mathrm{Kron}_{n_1,n_2}=\{(\alpha,\beta,\gamma)\text{ s.t. }\ell(\alpha)\leq n_1, \, \ell(\beta)\leq n_2, \, \ell(\gamma)\leq n_1n_2\text{ and }g_{\alpha,\beta,\gamma}\neq 0\} \]
is called the Kronecker semigroup.
\end{de}

\begin{rmk}
In the previous definition, it is not a restriction to bound the length of the third partition ($\gamma$, in our notations) by the product of the bounds of the first two. Indeed there is a well-known result concerning the Kronecker coefficients: if $g_{\alpha,\beta,\gamma}\neq 0$, then $\ell(\gamma)\leq\ell(\alpha)\ell(\beta)$. On the same topic, another usual property of the Kronecker coefficients that we may use in this article is that the value of $g_{\alpha,\beta,\gamma}$ does not depend on the order of the indexing partitions $\alpha$, $\beta$, and $\gamma$.
\end{rmk}

The classical result concerning $\mathrm{Kron}_{n_1,n_2}$ is then that it is a finitely generated semigroup. Thus we consider the cone generated by this semigroup:

\begin{de}
The set
\vspace{-5mm}
\begin{changemargin}{-6mm}{-6mm}
\[ \mathrm{PKron}_{n_1,n_2}=\{(\alpha,\beta,\gamma)\text{ s.t. }\ell(\alpha)\leq n_1, \, \ell(\beta)\leq n_2, \, \ell(\gamma)\leq n_1n_2\text{ and }\exists N\in\N^*,\; g_{N\alpha,N\beta,N\gamma}\neq 0\} \]
\end{changemargin}
is called the Kronecker cone, or the Kronecker polyhedron. It is a rational polyhedral cone.
\end{de}

The Kronecker coefficients are moreover known to possess an interesting stability property, discovered by F. Murnaghan in the 1930's: given three partitions $\lambda,\mu,\nu$, the sequence $\big(g_{\lambda+(d),\mu+(d),\nu+(d)}\big)_{d\in\N}$ is known to be constant when $d\gg 0$. In order to generalise this property, J. Stembridge defined in \cite{stembridge} the notion of a stable triple of partitions:

\begin{de}
A triple $(\alpha,\beta,\gamma)$ of partitions is said to be stable if $g_{\alpha,\beta,\gamma}\neq 0$ and, for any triple $(\lambda,\mu,\nu)$ of partitions, the sequence of general term $g_{\lambda+d\alpha,\mu+d\beta,\nu+d\gamma}$ is constant when $d\gg 0$.
\end{de}

With this terminology, Murnaghan's stability just states that the triple $\big((1),(1),(1)\big)$ is stable. An interesting characterisation of this stability notion has been proven by the work of Stembridge (in \cite{stembridge}) and S. Sam and A. Snowden (in \cite{sam-snowden})\footnote{Other proofs of this by different methods exist in \cite{paradan} and \cite{article1}.}:

\begin{prop}
A triple $(\alpha,\beta,\gamma)$ of partitions is stable if and only if, for any positive integer $d$, $g_{d\alpha,d\beta,d\gamma}=1$.
\end{prop}

This in particular leads us to define another notion close to this one:

\begin{de}
A triple $(\alpha,\beta,\gamma)\in\mathrm{PKron}_{n_1,n_2}$ is said to be almost stable if, for any positive integer $d$, $g_{d\alpha,d\beta,d\gamma}\leq 1$.
\end{de}

Stable triples are in particular of interest because of the following well-known result, which can for instance be found in \cite[Proposition 2]{manivel2}:

\begin{prop}
The set of stable triples in $\mathrm{Kron}_{n_1,n_2}$ is the intersection of $\mathrm{Kron}_{n_1,n_2}$ with a union of faces of the Kronecker cone $\mathrm{PKron}_{n_1,n_2}$.
\end{prop}

As a consequence we want to find ways to produce such faces of $\mathrm{PKron}_{n_1,n_2}$, which contain only stable triples, or possibly almost stable ones. There already exists one result in this direction, proven independently by L. Manivel (in \cite{manivel}) and E. Vallejo (in \cite{vallejo}), expressed in terms of additive matrices.

\begin{de}
A matrix $A=(a_{i,j})_{i,j}\in\mathcal{M}_{n_1,n_2}(\Z_{\geq 0})$ is said to be additive if there exist integers $x_1>\dots>x_{n_1}$ and $y_1>\dots>y_{n_2}$ such that, for all $(i,j),(k,l)\in\llbracket 1,n_1\rrbracket\times\llbracket 1,n_2\rrbracket$,
\[ a_{i,j}>a_{k,l}\Longrightarrow x_i+y_j>x_k+y_l. \]
\end{de}

The result of Manivel and Vallejo is then that any such additive matrix gives an explicit face of $\mathrm{PKron}_{n_1,n_2}$ which contains only stable triples. This face is moreover regular, which means that it contains some triple $(\alpha,\beta,\gamma)$ of regular partitions (i.e. $\alpha$, $\beta$, and $\gamma$ have respectively $n_1$, $n_2$, and $n_1n_2$ pairwise distinct parts, with the last one being possibly 0), and it has the minimal dimension possible for a regular face: $n_1n_2$.

\vspace{5mm}

In this article we obtain results producing, from an additive matrix, more faces of this kind. Actually, rather than looking precisely at an additive matrix, we look instead at what we call the order matrix, which sort of ``encodes the type'' of the additive matrix: considering an additive matrix $A=(a_{i,j})_{i,j}$ whose coefficients are pairwise distinct\footnote{This assumption is in fact not a restriction on the faces that we produce at the end.}, instead of the coefficients of $A$ we write their rank in the decreasingly-ordered sequence of the $a_{i,j}$'s. Then our first result (see Section \ref{cas cardinal 1}) is:

\begin{theo}
Any configuration of the following type in the order matrix:
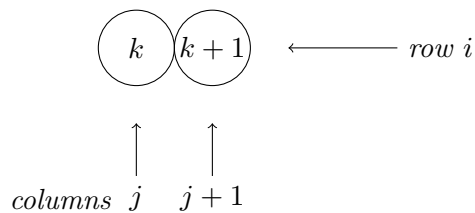
\begin{center}
\begin{tikzpicture}
\node (un) at (0,0) {$k$};
\draw (un) circle (5mm);
\node (deux) at (1,0) {$k+1$};
\draw (deux) circle (5mm);
\node (trois) at (4,0) {row $i$};
\draw[->] (trois) -- (2,0);
\node (quatre) at (0,-2) {$j$};
\draw[->] (quatre) -- (0,-1);
\node (cinq) at (1,-2) {$j+1$};
\draw[->] (cinq) -- (1,-1);
\node at (-1,-2) {columns};
\end{tikzpicture}
\end{center}
gives an explicit regular face of the Kronecker cone $\mathrm{PKron}_{n_1,n_2}$, of dimension $n_1n_2$, containing only stable triples.
The same result is true for each configuration of the type
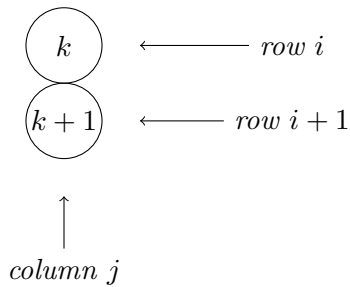
\begin{center}
\begin{tikzpicture}
\node (un) at (0,0) {$k$};
\draw (un) circle (5mm);
\node (deux) at (0,-1) {$k+1$};
\draw (deux) circle (5mm);
\node (trois) at (3,0) {row $i$};
\draw[->] (trois) -- (1,0);
\node (quatre) at (3,-1) {row $i+1$};
\draw[->] (quatre) -- (1,-1);
\node (cinq) at (0,-3) {column $j$};
\draw[->] (cinq) -- (0,-2);
\end{tikzpicture}
\end{center}
\end{theo}

Then we also obtain another result, concerning other types of configurations in the order matrix (called Configurations \circled{A} to \circled{E}, and involving now three or four coefficients of the order matrix, see Section \ref{cas cardinal 2}):

\begin{theo}
Each one of the Configurations {\normalfont\circled{A}} to {\normalfont\circled{E}} in the order matrix gives a face -- not necessarily regular and possibly reduced to zero -- of the Kronecker cone $\mathrm{PKron}_{n_1,n_2}$ which contains only almost stable triples.
\end{theo}

Obtaining these results is based on the notions of dominant and well-covering pairs, coming from the work of N. Ressayre, that we present in Section \ref{basics_well-covering_pairs}. At the end of this article (in Section \ref{application_small_examples}), we apply our results as well as Manivel and Vallejo's to all possible order matrices of small size (namely 2$\times$2, 3$\times$2, and 3$\times$3) in order to have a look at the number of new interesting faces of $\mathrm{PKron}_{n_1,n_2}$ that we can produce.

\vspace{5mm}

\textit{Acknowledgements:} I would like to thank Nicolas Ressayre for his advice and for invaluable discussions during the preparation of this article. I also acknowledge support from the French ANR (ANR project ANR-15-CE40-0012).

\section{Definitions and a few existing general results}\label{basics_well-covering_pairs}

\subsection{Definitions in the general context}\label{very_general_context}

For now $G$ is a connected complex reductive group acting on a smooth projective variety $X$. Let us consider a maximal torus $T$ in $G$, $\tau$ a one-parameter subgroup of $T$ (denoted by $\tau\in X_*(T)$), and $C$ an irreducible component of $X^\tau$, the set of points in $X$ fixed by $\tau$. We denote by $G^\tau$ the centraliser of $\tau$ (i.e. of $\im\tau$) in $G$ and set
\[ P(\tau)=\{g\in G\text{ s.t. }\lim\limits_{t\to 0}\tau(t)g\tau(t^{-1})\text{ exists}\}. \]
Notice that $P(\tau)$ is a parabolic subgroup of $G$ and that $G^\tau$ is the Levi subgroup of $P(\tau)$ containing $T$. Consider then
\[ C^+=\{x\in X\text{ s.t. }\lim\limits_{t\to 0}\tau(t).x\in C\}, \]
which is a $P(\tau)$-stable. For any $x\in C$, we define the following subspaces of $\mathrm{T}_x X$, the Zariski tangent space of $X$ at $x$:
\[ \begin{array}{l}
\mathrm{T}_x X_{>0}=\{\xi\in\mathrm{T}_x X\text{ s.t. }\lim\limits_{t\to 0}\tau(t).\xi=0\},\\
\mathrm{T}_x X_{<0}=\{\xi\in\mathrm{T}_x X\text{ s.t. }\lim\limits_{t\to 0}\tau(t^{-1}).\xi=0\},\\
\mathrm{T}_x X_0=(\mathrm{T}_x X)^\tau,\\
\mathrm{T}_x X_{\geq 0}=\mathrm{T}_x X_{>0}\oplus\mathrm{T}_x X_0,\\
\mathrm{T}_x X_{\leq 0}=\mathrm{T}_x X_{<0}\oplus\mathrm{T}_x X_0.
\end{array} \]

\begin{theo}[Bia\l{}ynicki-Birula]\label{bialynicki-birula_thm}$\hphantom{a}$\\
\vspace{-5mm}
\begin{itemize}
\item[(i)] $C$ is smooth and, for any $x\in C$, $\mathrm{T}_x C=\mathrm{T}_x X_0$.
\item[(ii)] $C^+$ is smooth and irreducible and, for any $x\in C$, $\mathrm{T}_x C^+=\mathrm{T}_x X_{\geq 0}$.
\end{itemize}
\end{theo}

\noindent We can now consider
\[ \begin{array}{rccl}
\eta: & G\times_{P(\tau)} C^+ & \longrightarrow & X\\
 & [g:x] & \longmapsto & g.x
\end{array}. \]
The following definition comes from \cite{ressayre}.

\begin{de}
The pair $(C,\tau)$ is said to be dominant if $\eta$ is, and covering if $\eta$ is birational. It is said to be well-covering when $\eta$ induces an isomorphism onto an open subset of $X$ intersecting $C$.
\end{de}

\subsection{In the context of Kronecker coefficients}

We consider from now on $V_1$ and $V_2$ two complex vector spaces of dimension respectively $n_1$ and $n_2$. We then set $G_1=\gl(V_1)$, $G_2=\gl(V_2)$, $G=G_1\times G_2$, $\hat{G}=\gl(V_1\otimes V_2)$. We also choose $T_1$, $T_2$, $T=T_1\times T_2$, and $\hat{T}\supset T$ respective maximal tori, and $B_1$, $B_2$, $B=B_1\times B_2$, $\hat{B}$ respective Borel subgroups containing the corresponding tori. We use moreover the following notation: if $V'$ is a complex vector space of dimension $n'$ and $B'$ a Borel subgroup of $\gl(V')$, then any $n'$-tuple of integers $\lambda$ gives a character of $B'$, and we denote by $\C_\lambda(B')$ the one-dimensional complex representation of $B'$ given by this character. Recall then a classical interpretation of the Kronecker coefficients (see for instance \cite[Section 2.1]{article1}):

\begin{prop}
If $\alpha$, $\beta$, and $\gamma$ are partitions of lengths at most $n_1$, $n_2$ and $n_1n_2$ respectively, set:
\[ \Lb_\alpha=G_1\times_{B_1}\C_{-\alpha}(B_1),\quad \Lb_\beta=G_2\times_{B_2}\C_{-\beta}(B_2),\quad \Lb_\gamma=\hat{G}\times_{\hat{B}}\C_{-\gamma}(\hat{B}), \]
which are line bundles on $G_1/B_1$, $G_2/B_2$, and $\hat{G}/\hat{B}$ respectively. Then one has a $G$-linearised line bundle $\Lb_{\alpha,\beta,\gamma}=\Lb_\alpha\otimes\Lb_\beta\otimes\Lb^*_\gamma$ on $G/B\times\hat{G}/\hat{B}$ such that:
\[ g_{\alpha,\beta,\gamma}=\dim\h^0(G/B\times\hat{G}/\hat{B},\Lb_{\alpha,\beta,\gamma})^G. \]
\end{prop}

We consider in addition parabolic subgroups $P$ of $G$ and $\hat{P}$ of $\hat{G}$ containing the Borel subgroups. All corresponding Lie algebras will be denoted with lower case gothic letters. We consider
\[ X=G/P\times\hat{G}/\hat{P}, \]
on which $G$ acts diagonally. We finally denote by $W$ and $\hat{W}$ the Weyl groups associated to $G$ and $\hat{G}$ respectively, and by $W_P$ (resp.$\hat{W}_{\hat{P}}$) the Weyl group of the Levi subgroup of $P$ (resp. $\hat{P}$) containing $T$ (resp. $\hat{T}$). It is canonically a subgroup of $W$ (resp. $\hat{W}$).

\vspace{5mm}

We also give notations concerning the root systems: let us denote by $\Phi$ (resp. $\hat{\Phi}$) the set of roots of $G$ (resp. $\hat{G}$), with $\Phi^+$ and $\Phi^-$ (resp. $\hat{\Phi}^+$ and $\hat{\Phi}^-$) the subsets of positive and negative ones with respect to the choice of $B$ (resp. $\hat{B}$). Finally, the set of roots of $\hat{\mathfrak{p}}$ is denoted by $\hat{\Phi}_{\hat{\mathfrak{p}}}$.

\vspace{5mm}

Let us consider $\tau\in X_*(T)$. It is known that the irreducible components of $X^\tau$ are the $G^\tau v^{-1}P/P\times\hat{G}^\tau\hat{v}^{-1}\hat{P}/\hat{P}$, for $v\in W_{P}\backslash W/W_{P(\tau)}$ and $\hat{v}\in \hat{W}_{\hat{P}}\backslash \hat{W}/\hat{W}_{\hat{P}(\tau)}$. We then fix two such $v$ and $\hat{v}$, and denote by $C$ the corresponding irreducible component of $X^\tau$. Therefore, if $(\alpha,\beta,\gamma)$ is a triple of partitions such that $\Lb_{\alpha,\beta,\gamma}$ descends to a line bundle on $X$ -- that we will also denote by $\Lb_{\alpha,\beta,\gamma}$ --\footnote{This corresponds to a simple condition concerning the forms of the partitions, which must correspond to the types of the partial flag varieties $G/P$ and $\hat{G}/\hat{P}$.} then, for any $x\in C$, $\C^*$ acts via $\tau$ on the fibre $(\Lb_{\alpha,\beta,\gamma})_x$ over $x$. This action is given by an integer $n$ which, since $C$ is an irreducible component, does not depend on $x\in C$. We then set $\mu^{\Lb_{\alpha,\beta,\gamma}}(C,\tau)=-n$.

\begin{lemma}
For any dominant pair $(C,\tau)$ we consider the set of all  triples $(\alpha,\beta,\gamma)\in\mathrm{PKron}_{n_1,n_2}$ such that $\mu^{\Lb_{\alpha,\beta,\gamma}}(C,\tau)=0$. Then it is a face of $\mathrm{PKron}_{n_1,n_2}$ (possibly reduced to zero). Moreover it can also be described as:
\[ \{(\alpha,\beta,\gamma)\text{ s.t. }X^{ss}(\Lb_{\alpha,\beta,\gamma})\cap C\neq\emptyset\}. \]
We denote this face by $\mathcal{F}(C)$.
\end{lemma}

Note that this result is actually valid in the general context of Section \ref{very_general_context}, where $X$ is any smooth projective variety and $G$ is a connected complex reductive group acting on $X$. The equivalent of $\mathrm{PKron}_{n_1,n_2}$ is then the cone $\big\{\Lb\text{ s.t. }\exists N\in\N^*, \; \h^0(X,\Lb^{\otimes N})^G\neq\{0\}\big\}$.

\begin{proof}
It comes directly from \cite[Lemma 3]{ressayre}.
\end{proof}

\begin{lemma}\label{lemma_uniqueness}
If $P=B$ or $\hat{P}=\hat{B}$, and if $(C_1,\tau_1)$ and $(C_2,\tau_2)$ are two well-covering pairs such that $\mathcal{F}(C_1)=\mathcal{F}(C_2)$, then there exists $g\in G$ such that $g.C_2=C_1$.
\end{lemma}

\begin{proof}
This comes from \cite[Lemma 6.5]{ressayre3}.
\end{proof}

Ressayre also proved, in \cite{ressayre}, that any regular face of $\mathrm{PKron}_{n_1,n_2}$ is given by a well-covering pair\footnote{He even proved it in a much more general setting than $\mathrm{PKron}_{n_1,n_2}$.}. Finally, another consequence of \cite{ressayre3} is that, if $C$ is a singleton and $(C,\tau)$ is well-covering, then the face $\mathcal{F}(C)$ is a regular face of minimal dimension of $\mathrm{PKron}_{n_1,n_2}$ (i.e. $n_1n_2$).

\section{Application to obtain stable triples}

\subsection{Link between well-covering pairs and stability}\label{result_kronecker_section}

\begin{theo}
Assume that $(C,\tau)$ is well-covering. Then, for all $G$-linearised line bundles $\Lb$ on $X$ such that $\mu^\Lb(C,\tau)=0$,
\[ \h^0(X,\Lb)^G\simeq\h^0(C,\left.\Lb\right|_C)^{G^\tau}. \]
\end{theo}

\begin{proof}
It is Theorem 4 of \cite{ressayre}.
\end{proof}

\begin{rmk}\label{remark dominant}
If we only make the hypothesis that $(C,\tau)$ is dominant (and $\mu^\Lb(C,\tau)=0$), we still have that
\[ \h^0(X,\Lb)^G\hookrightarrow\h^0(C,\left.\Lb\right|_C)^{G^\tau}. \]
\end{rmk}

\begin{cor}\label{result_stable_triple}
Assume that $(C,\tau)$ is well-covering, and that $G^\tau$ has a dense orbit in $C$. Then the face $\mathcal{F}(C)$ contains only almost stable triples.
\end{cor}

\begin{proof}
It is an immediate consequence of the previous theorem: let $(\alpha,\beta,\gamma)\in\mathcal{F}(C)$. Then $\mu^{\Lb_{\alpha,\beta,\gamma}}(C,\tau)=0$ and therefore
\[ \forall d\in\Z_{>0}, \; g_{d\alpha,d\beta,d\gamma}=\dim\h^0(X,\Lb^{\otimes d}_{\alpha,\beta,\gamma})^G=\dim\h^0(C,\left.\Lb\right|_C^{\otimes d})^{G^\tau}\leq 1 \]
since $G^\tau$ has a dense orbit in $C$.
\end{proof}

\begin{rmk}
If we only make the hypothesis that $(C,\tau)$ is dominant, Remark \ref{remark dominant} tells us that we still have almost stable triples. But note that $\mathcal{F}(C)$ can be reduced to zero.
\end{rmk}

\begin{rmk}\label{remark_dominant_regular_case}
There is an important particular case when $\tau$ is dominant, regular (i.e. for all $\alpha\in\Phi$, $\langle\alpha,\tau\rangle\neq 0$), and $\hat{G}$-regular (i.e. for all $\hat{\alpha}\in\hat{\Phi}$, $\langle\hat{\alpha},\tau\rangle\neq 0$). Then $G^\tau=T$ and $\hat{G}^\tau=\hat{T}$. As a consequence, $C$ is a singleton -- say $\{x_0\}$ --, and the condition ``$G^\tau$ has a dense orbit in $C$'' is automatic. Moreover one then has:
\[ \dim\h^0(C,\left.\Lb\right|_C)^{G^\tau}=1\Longleftrightarrow T\text{ acts trivially on }\Lb_{x_0}\Longleftrightarrow\forall\sigma\in X_*(T), \: \mu^\Lb(C,\sigma)=0. \]
\end{rmk}

\noindent All the previous results and remarks lead directly to the following main result:

\begin{theo}\label{result_tau_dominant}
Assume that $(C,\tau)$ is well-covering and that $\tau$ is dominant, regular, and $\hat{G}$-regular (then $C$ is a singleton). Then $\mathcal{F}(C)$ is a regular face of minimal dimension (i.e. $n_1n_2$) of the Kronecker cone $\mathrm{PKron}_{n_1,n_2}$ and contains only stable triples.
\end{theo}

\begin{rmk}\label{schubert_condition_remark}
When $X$ has this form, there is a characterisation of the dominant and covering pairs with a Schubert condition. It can be found in \cite{ressayre}. Let us explain it quickly here: we use the cohomology ring, $\h^*(G/P(\tau),\Z)$, of $G/P(\tau)$ and we denote, for any closed subvariety $Y$ of $G/P(\tau)$, by $[Y]\in\h^*(G/P(\tau),\Z)$ its cycle class in cohomology. We also use, with the same notations, $\h^*(\hat{G}/\hat{P}(\tau),\Z)$.\\
Then, since $P(\tau)=G\cap\hat{P}(\tau)$, $G/P(\tau)$ identifies with the $G$-orbit of $\hat{P}(\tau)/\hat{P}(\tau)$ in $\hat{G}/\hat{P}(\tau)$, which gives a closed immersion $\iota:G/P(\tau)\hookrightarrow\hat{G}/\hat{P}(\tau)$. It induces a map $\iota^*$ in cohomology:
\[ \iota^*:\h^*(\hat{G}/\hat{P}(\tau),\Z)\longrightarrow\h^*(G/P(\tau),\Z). \]
Then the result from \cite{ressayre} (Lemma 14) is:
\begin{lemma}\label{schubert_condition_lemma}
\begin{itemize}
\item[(i)] The pair $(C,\tau)$ is dominant if and only if
\[ [\overline{PvP(\tau)/P(\tau)}]\cdot \iota^*([\overline{\hat{P}\hat{v}\hat{P}(\tau)/\hat{P}(\tau)}])\neq 0. \]
\item[(ii)] It is covering if and only if
\[ [\overline{PvP(\tau)/P(\tau)}]\cdot \iota^*([\overline{\hat{P}\hat{v}\hat{P}(\tau)/\hat{P}(\tau)}])=[\mathrm{pt}], \]
i.e. if and only if the intersection between two generic translates in $\hat{G}/\hat{P}(\tau)$ of $\overline{PvP(\tau)/P(\tau)}$ and $\overline{\hat{P}\hat{v}\hat{P}(\tau)/\hat{P}(\tau)}$ contains exactly one point.
\end{itemize}
\end{lemma}
\end{rmk}

\subsection{A sufficient condition to get dominant pairs}

We now want to see that we can indeed obtain dominant or well-covering pairs $(C,\tau)$. For this we consider respective bases of the vector spaces $V_1$ and $V_2$: $(e_1,\dots,e_{n_1})$ and $(f_1,\dots,f_{n_2})$. They give also a basis of $V_1\otimes V_2$: $(e_i\otimes f_j)_{i,j}$ ordered lexicographically (i.e. $(e_1\otimes f_1,e_1\otimes f_2,\dots,e_{n_1}\otimes f_{n_2})$), sometimes denoted $(\hat{e}_1,\dots,\hat{e}_{n_1n_2})$\footnote{\label{footnote ordre}The ordering of the basis of $V_1\otimes V_2$ gives in particular an explicit bijection between $\llbracket 1,n_1\rrbracket\times\llbracket 1,n_2\rrbracket$ and $\llbracket 1,n_1n_2\rrbracket$, which we will regularly use to identify the two in what follows.}. Thanks to these bases we will often identify $G_1$, $G_2$, and $\hat{G}$ respectively with $\gl_{n_1}(\C)$, $\gl_{n_2}(\C)$, and $\gl_{n_1n_2}(\C)$. We finally take $B$ and $\hat{B}$ the respective Borel subgroups of $G$ and $\hat{G}$ formed by the upper-triangular matrices, and set from now on $X=G/B\times\hat{G}/\hat{B}$.

\vspace{5mm}

Start now from a one-parameter subgroup $\tau$ of $T$ which is supposed to be dominant, regular, and even $\hat{G}$-regular. In particular, $\tau$ has the form
\[ \begin{array}{rccl}
\tau: & \C^* & \longrightarrow & T\\
 & t & \longmapsto & (\begin{pmatrix}
t^{x_1} && \\
 & \ddots & \\
 && t^{x_{n_1}}
\end{pmatrix},\begin{pmatrix}
t^{y_1} && \\
 & \ddots & \\
 && t^{y_{n_2}}
\end{pmatrix})
\end{array}, \]
with non-negative integers $x_1>\dots>x_{n_1}$, $y_1>\dots>y_{n_2}$. We then create the matrix $M=(x_i+y_j)_{i,j}\in\mm_{n_1,n_2}(\R)$. Since $\tau$ was taken $\hat{G}$-regular, it has the property of having its coefficients which are pairwise distinct. From this we define what we will call the ``order matrix'' of $\tau$: it is a matrix having the same size as $M$ but whose coefficient at position $(i,j)$ is the ranking of the coefficient $x_i+y_j$ when one orders the coefficients of $M$ decreasingly (we will usually circle that ranking when we write the order matrix in order to highlight the difference with the coefficient $x_i+y_j$).

\vspace{5mm}

Giving such an order matrix is equivalent to giving a flag $\hat{w}.\hat{B}/\hat{B}$ fixed by $\hat{T}$ (and then a well-defined $\hat{w}\in\hat{W}$): to each matrix position $(i,j)\in\llbracket 1,n_1\rrbracket\times\llbracket 1,n_2\rrbracket$ we associate the element $e_i\otimes f_j$ of the basis of $V_1\otimes V_2$. Then we create a $\hat{T}$-stable complete flag in $V_1\otimes V_2$ by ordering the elements $e_i\otimes f_j$ according to the numbers in the order matrix.

\paragraph{Example:} For the one-parameter subgroup
\[ \begin{array}{rccl}
\tau: & \C^* & \longrightarrow & T\\
 & t & \longmapsto & (\begin{pmatrix}
t^4 &&\\
& t^2 &\\
&& 1
\end{pmatrix},\begin{pmatrix}
t^3 &\\
& 1
\end{pmatrix})
\end{array}, \]
the matrix $M$ is $\begin{pmatrix}
7&4\\
5&2\\
3&0
\end{pmatrix}$, and so the order matrix of $\tau$ is
\[ \begin{pmatrix}
\text{\ding{192}}&\text{\ding{194}}\\
\text{\ding{193}}&\text{\ding{196}}\\
\text{\ding{195}}&\text{\ding{197}}
\end{pmatrix}. \]
Then the flag $\hat{w}.\hat{B}/\hat{B}$ happens to be
\[ \begin{array}{l}
\big(\C e_1\otimes f_1\\
\subset\C e_1\otimes f_1\oplus\C e_2\otimes f_1\\
\subset\C e_1\otimes f_1\oplus\C e_2\otimes f_1\oplus\C e_1\otimes f_2\\
\subset\C e_1\otimes f_1\oplus\C e_2\otimes f_1\oplus\C e_1\otimes f_2\oplus\C e_3\otimes f_1\\
\subset\C e_1\otimes f_1\oplus\C e_2\otimes f_1\oplus\C e_1\otimes f_2\oplus\C e_3\otimes f_1\oplus\C e_2\otimes f_2\\
\subset\C e_1\otimes f_1\oplus\C e_2\otimes f_1\oplus\C e_1\otimes f_2\oplus\C e_3\otimes f_1\oplus\C e_2\otimes f_2\oplus\C e_3\otimes f_2\big),
\end{array} \]
which we denote by
\[\mathrm{fl}(e_1\otimes f_1,e_2\otimes f_1,e_1\otimes f_2,e_3\otimes f_1,e_2\otimes f_2,e_3\otimes f_2)\in\fl(V_1\otimes V_2). \]
This corresponds to
\[ \hat{w}=\begin{pmatrix}
1&2&3&4&5&6\\
1&3&2&5&4&6
\end{pmatrix}=\underset{\text{notation as a product of transpositions}}{\underbrace{\big(2\hphantom{a}3\big)\big(4\hphantom{a}5\big)}}\in\gs_6\simeq\hat{W}. \]

\vspace{5mm}

As usual with such a one-parameter subgroup, we get two parabolic subgroups $P(\tau)$ and $\hat{P}(\tau)$ of $G$ and $\hat{G}$ respectively. According to the hypotheses made on $\tau$, $P(\tau)=B$ in that case, and $\hat{P}(\tau)$ is a Borel subgroup denoted instead $\hat{B}(\tau)$. The set of positive (resp. negative) roots of $\hat{G}$ for this choice of Borel subgroup is denoted by $\hat{\Phi}^+(\tau)$ (resp. $\hat{\Phi}^-(\tau)$). The unipotent radicals of $B$, $\hat{B}$, and $\hat{B}(\tau)$ are respectively denoted $U$, $\hat{U}$, and $\hat{U}(\tau)$, while those of the respective opposite Borel subgroups will be $U^-$, $\hat{U}^-$, $\hat{U}^-(\tau)$.

\vspace{5mm}

Consider now two elements $v\in W$ and $\hat{v}\in\hat{W}$. They give
\[ C=\{x_0\}=\{(v^{-1}B/B,\hat{v}^{-1}\hat{B}/\hat{B})\}, \]
an irreducible component of $X^T$. As usual we then have
\[ C^+=Bv^{-1}B/B\times\hat{B}(\tau)\hat{v}^{-1}\hat{B}/\hat{B} \]
and
\[ \begin{array}{rccl}
\eta: & G\times_B C^+ & \longrightarrow & X\\
 & [g:x] & \longmapsto & g.x
\end{array}. \]
We need two more notations: let us denote by $\rho$ the morphism of restriction of roots of $\hat{G}$ (which are morphisms from $\hat{T}$ to $\C$) to morphisms from $T$ to $\C$ and, for all $u\in W$ (resp. $\hat{u}\in\hat{W}$), set $\Phi(u)=\Phi^-\cap u\Phi^+$ (resp. $\hat{\Phi}(\hat{u})=\hat{\Phi}^-\cap\hat{u}\hat{\Phi}^+$).

\begin{prop}\label{condition dominant pair}
If $\rho\left(\hat{w}\hat{\Phi}\big(((\hat{v}\hat{w})^\vee)^{-1}\big)\right)\subset\Phi(v^{-1})$ and $\rho:\hat{w}\hat{\Phi}\big(((\hat{v}\hat{w})^\vee)^{-1}\big)\longrightarrow\Phi(v^{-1})$ is a bijection, then the pair $(C,\tau)$ is dominant.
\end{prop}

\begin{proof}
The map $\eta$ will be dominant if its tangent map at $[e:x_0]$, $\mathrm{T}_{[e:x_0]}\eta$, is an isomorphism. Moreover, on root spaces for non-negative weights (i.e. on $\mathrm{T}_{x_0}C^+=\mathrm{T}_{x_0}X_{\geq 0}$), $\mathrm{T}_{[e:x_0]}\eta$ is just the identity (cf Theorem \ref{bialynicki-birula_thm}). As a consequence,
\[ \begin{array}{rl}
 & \mathrm{T}_{[e:x_0]}\eta\text{ is an isomorphism}\\
\Longleftrightarrow & \mathrm{T}_{[e:x_0]}\eta|_{\mathfrak{u}^-}:\mathfrak{u}^-\longrightarrow\mathfrak{u}^-\cap v^{-1}\mathfrak{u}^-v\oplus\hat{\mathfrak{u}}^-(\tau)\cap\hat{v}^{-1}\hat{\mathfrak{u}}^-\hat{v}\text{ is an isomorphism.}
\end{array} \]
Then, if we define
\[ \begin{array}{rccl}
\mathrm{orb}: & U^- & \longrightarrow & X\\
 & u & \longmapsto & u.x_0
\end{array}, \]
we have $\mathrm{T}_{[e:x_0]}\eta|_{\mathfrak{u}^-}=\mathrm{T}_e\mathrm{orb}$. Moreover $\mathrm{T}_e\mathrm{orb}$ is an isomorphism if and only if it is injective, i.e. if and only if the isotropy subgroup $U_{x_0}^-$ of $x_0$ in $U^-$ is finite. As a consequence, $\mathrm{T}_e\mathrm{orb}$ is an isomorphism if and only if the Lie algebra of $U_{x_0}^-$ is $\{0\}$. Therefore,
\[ \begin{array}{cl}
 & \mathrm{T}_{[e:x_0]}\eta\text{ is an isomorphism}\\
\Longleftrightarrow & \mathfrak{u}^-\simeq\mathfrak{u}^-\cap v^{-1}\mathfrak{u}^-v\oplus\hat{\mathfrak{u}}^-(\tau)\cap\hat{v}^{-1}\hat{\mathfrak{u}}^-\hat{v}\text{ as }T\text{-modules}\\
\Longleftrightarrow & \dps\bigoplus_{\beta\in\Phi^-\cap v^{-1}\Phi^+}\mathfrak{g}_\beta\simeq\dps\bigoplus_{\hat{\beta}\in\hat{\Phi}^-(\tau)\cap\hat{v}^{-1}\hat{\Phi}^-} \mathfrak{g}_{\hat{\beta}}\text{ as }T\text{-modules}.
\end{array} \]
Notice now that one has $\hat{\Phi}^-(\tau)=\hat{w}\hat{\Phi}^-$, where $\hat{w}\in\hat{W}$ is the element coming from the order matrix, as explained above. Then $\hat{\Phi}^-(\tau)\cap\hat{v}^{-1}\hat{\Phi}^-=\hat{w}\left(\hat{\Phi}^-\cap((\hat{v}\hat{w})^\vee)^{-1}\hat{\Phi}^+\right)$, denoted\footnote{for any $\hat{u}\in\hat{W}$, $\hat{u}^\vee$ is defined as $\hat{w}_0\hat{u}$, where $\hat{w}_0$ is the longest element of the Weyl group $\hat{W}$} by $\hat{w}\hat{\Phi}\big((\hat{v}\hat{w})^\vee\big)$. Thus:
\vspace{-5mm}
\begin{changemargin}{-9mm}{-9mm}
\[ \begin{array}{rl}
 & \mathrm{T}_{[e:x_0]}\eta\text{ is an isomorphism}\\
\Longleftrightarrow & \rho\left(\hat{\Phi}^-(\tau)\cap\hat{v}^{-1}\hat{\Phi}^-\right)\subset\Phi^-\cap v^{-1}\Phi^+\text{ and }\rho:\hat{\Phi}^-(\tau)\cap\hat{v}^{-1}\hat{\Phi}^-\longrightarrow\Phi^-\cap v^{-1}\Phi^+\text{ is a bijection}\\
\Longleftrightarrow & \rho\left(\hat{w}\hat{\Phi}\big(((\hat{v}\hat{w})^\vee)^{-1}\big)\right)\subset\Phi(v^{-1})\text{ and }\rho:\hat{w}\hat{\Phi}\big(((\hat{v}\hat{w})^\vee)^{-1}\big)\longrightarrow\Phi(v^{-1})\text{ is a bijection,}
\end{array} \]
\end{changemargin}
\end{proof}

\begin{rmk}
It is a classical result that, for $u\in W$ (resp. $\hat{u}\in\hat{W}$), the cardinal of the set $\Phi(u)$ (resp. $\hat{\Phi}(\hat{u})$) corresponds to the length of the element $u$ (resp. $\hat{u}$) of the Coxeter group $W$ (resp. $\hat{W}$), denoted $\ell(u)$ (resp. $\ell(\hat{u})$). As a consequence $\sharp\Phi(v^{-1})=\ell(v^{-1})=\ell(v)$.
\end{rmk}

In this context ($C=\{(v^{-1}B/B,\hat{v}^{-1}\hat{B}/\hat{B})\}$), there is a characterisation of well-covering pairs given by Ressayre in \cite{ressayre}, Proposition 11:

\begin{lemma}
The pair $(C,\tau)$ is well-covering if and only if it is covering and
\[ v^{-1}.\left(\sum_{\alpha\in\Phi^-\cap v\Phi^-}\alpha\right)+\rho\left(\hat{v}^{-1}.\sum_{\hat{\alpha}\in\hat{\Phi}^-\cap\hat{v}\hat{\Phi}^-(\tau)}\hat{\alpha}\right)=\sum_{\alpha\in\Phi^-}\alpha. \]
\end{lemma}

\begin{lemma}\label{lemma_covering_well}
If $v$ and $\hat{v}$ are chosen as in Proposition \ref{condition dominant pair}, then:
\[ (C,\tau)\text{ is covering}\quad\Longrightarrow\quad(C,\tau)\text{ is well-covering}. \]
\end{lemma}

\begin{proof}
Assume that $v$ and $\hat{v}$ are chosen as in Proposition \ref{condition dominant pair} and that $(C,\tau)$ is covering. Then:
\[ \begin{array}{cl}
 & v^{-1}.\left(\dps\sum_{\alpha\in\Phi^-\cap v\Phi^-}\alpha\right)+\rho\left(\hat{v}^{-1}.\dps\sum_{\hat{\alpha}\in\hat{\Phi}^-\cap\hat{v}\hat{\Phi}^-(\tau)}\hat{\alpha}\right)-\dps\sum_{\alpha\in\Phi^-}\alpha\\
= & \dps\sum_{\alpha\in\Phi^-\cap v^{-1}\Phi^-}\alpha-\sum_{\alpha\in\Phi^-}\alpha+\rho\left(\dps\sum_{\hat{\alpha}\in\hat{\Phi}^-(\tau)\cap \hat{v}^{-1}\hat{\Phi}^-}\hat{\alpha}\right)\\
= & -\dps\sum_{\alpha\in\Phi^-\cap v^{-1}\Phi^+}\alpha+\dps\sum_{\alpha\in\Phi^-\cap v^{-1}\Phi^+}\alpha\\
= & 0
\end{array} \]
and, by the previous lemma, $(C,\tau)$ is well-covering.
\end{proof}

\subsection{First case to apply Proposition 3.9: an existing result}

The first and simplest case to satisfy the condition of Proposition \ref{condition dominant pair} is the case when $\sharp\Phi(v^{-1})=0$, i.e. $\ell(v)=0$. This means that $v=\mathsf{1}_W$, the unit in $W$. But then we also need that $\sharp\hat{\Phi}\big(((\hat{v}\hat{w})^\vee)^{-1}\big)=0$, i.e. $\hat{w}^{-1}\hat{v}^{-1}\hat{w}_0=((\hat{v}\hat{w})^\vee)^{-1}=\mathsf{1}_{\hat{W}}$. Thus $\hat{v}^{-1}=\hat{w}\hat{w}_0$ gives a dominant pair $(C,\tau)$.

\vspace{5mm}

Moreover the Schubert condition given in Lemma \ref{schubert_condition_lemma} is not difficult to check here: according to the form of $C=\{(B/B,\hat{w}\hat{w}_0\hat{B}/\hat{B})\}$, the first Schubert variety to consider is just $\overline{B/B}$, which is a single point, whereas the second one is $\overline{\hat{B}\hat{w}_0\hat{w}^{-1}\hat{B}(\tau)/\hat{B}(\tau)}=\overline{\hat{B}\hat{w}_0\hat{B}/\hat{B}}$, which is the whole variety $\hat{G}/\hat{B}$. Hence the product of the two Schubert classes is in fact the class of a point, and then $(C,\tau)$ is well-covering by Lemma \ref{schubert_condition_lemma} and Lemma \ref{lemma_covering_well}. We obtain the following:

\begin{theo}\label{thm_length_zero}
Each order matrix corresponding to a dominant, regular, $\hat{G}$-regular one-parameter subgroup $\tau$ of $T$ gives a well-covering pair $(C,\tau)$ with:
\[ C=\big\{(B/B,\hat{w}\hat{w}_0\hat{B}/\hat{B})\big\}. \]
As a consequence the corresponding face $\mathcal{F}(C)$ of the Kronecker cone $\mathrm{PKron}_{n_1,n_2}$ contains only stable triples.
\end{theo}

This theorem is actually an already existing result, evoked in the introduction and due independently to L. Manivel (see \cite{manivel}) and E. Vallejo (see \cite{vallejo}). Let us now explain their result and why it corresponds to the previous one.

\vspace{5mm}

We consider a matrix $A=(a_{i,j})_{i,j}\in\mathcal{M}_{n_1,n_2}(\Z_{\geq 0})$. We call $\lambda$ and $\mu$ its 1-marginals, i.e. the finite sequence of integers $(\lambda_1,\dots,\lambda_{n_1})$ and $(\mu_1,\dots,\mu_{n_2})$ given by
\[ \lambda_i=\sum_{j=1}^{n_2} a_{i,j}\qquad\text{and}\qquad\mu_j=\sum_{i=1}^{n_1} a_{i,j}, \]
and we suppose that $\lambda$ and $\mu$ are partitions (i.e. are non-increasing). Moreover, we denote by $\nu$ the $\pi$-sequence of $A$, i.e. the (finite) non-increasing sequence $(\nu_1,\dots,\nu_{n_1n_2})$ formed by the entries of $A$.

\begin{de}
The matrix $A$ is said to be additive if there exist integers $x_1>\dots>x_{n_1}$ and $y_1>\dots>y_{n_2}$ such that, for all $(i,j),(k,l)\in\llbracket 1,n_1\rrbracket\times\llbracket 1,n_2\rrbracket$,
\[ a_{i,j}>a_{k,l}\Longrightarrow x_i+y_j>x_k+y_l. \]
\end{de}

\begin{theo}[Manivel, Vallejo]\label{result_manivel_vallejo}
Assume that the matrix $A$ is additive. Then the triple $(\lambda,\mu,\nu)$ of partitions is a stable triple.
\end{theo}

Manivel and Vallejo gave different proofs of this result. What we want to highlight here is that it corresponds to Theorem \ref{thm_length_zero}.

\begin{proof}
The parabolic subgroup $\hat{P}$ of $\hat{G}$ we consider this time is the one corresponding to the ``shape'' of $-\hat{w}_0.\nu$, i.e. the one such that $\Lb^*_\nu$ is the pull-back of an ample line bundle on $\hat{G}/\hat{P}$. Furthermore, we take $P=B$, and so
\[ Y=G/B\times\hat{G}/\hat{P}. \]
The matrix $A$ gives a flag in $\hat{G}/\hat{P}$, similarly to what we explained about the order matrix of a one-parameter subgroup of $T$: the ordering of the coefficients $a_{i,j}$ in \textit{non-decreasing} order (it is different from before) gives a partial (since some of these coefficients can be equal) ordering of the elements $e_i\otimes f_j$ of the basis of $V_1\otimes V_2$. Then this ordering corresponds to a $\hat{T}$-stable partial flag in $V_1\otimes V_2$ that is precisely an element of $\hat{G}/\hat{P}$.

\vspace{5mm}

\noindent \underline{Example:} The additive matrix $\begin{pmatrix}
3&2\\
3&1
\end{pmatrix}$ gives the flag $(\C e_2\otimes f_2\subset\C e_2\otimes f_2\oplus\C e_1\otimes f_2\subset\C e_2\otimes f_2\oplus\C e_1\otimes f_2\oplus\C e_1\otimes f_1\oplus\C e_2\otimes f_1=V_1\otimes V_2)\in\fl(1,2;V_1\otimes V_2)$, which we will denote by $\mathrm{fl}\big(e_2\otimes f_2,e_1\otimes f_2,\{e_1\otimes f_1,e_2\otimes f_1\}\big)$.

\vspace{5mm}

\noindent The obtained flag is thus of the form $\hat{u}\hat{P}/\hat{P}$, with $\hat{u}\in\hat{W}\simeq\gs_{n_1n_2}$. In the previous example, the flag can for instance be written with $\hat{u}=(1\hphantom{a}4\hphantom{a}3).$

\vspace{5mm}

\noindent\underline{Remark:} This element $\hat{u}$ is in general not uniquely defined: what is unique is its class in $\hat{W}/\hat{W}_{\hat{P}}$, but it is sufficient to pick one representative $\hat{u}\in\hat{W}$ of this one.

\vspace{5mm}

\noindent We then set
\[ x_0=(B/B,\hat{u}\hat{P}/\hat{P})\in Y. \]
The point $x_0$ is fixed by $T$, and we can check (this is an easy computation) that, for any one-parameter subgroup $\tau$ of $T$,
\[ \mu^{\Lb_{\lambda,\mu,\nu}}(x_0,\tau)=0. \]
Since $A$ is additive, there exist integers $x_1>\dots>x_m$ and $y_1>\dots>y_n$ such that, for all $(i,j),(k,l)\in\llbracket 1,m\rrbracket\times\llbracket 1,n\rrbracket$,
\[ a_{i,j}>a_{k,l}\Longrightarrow x_i+y_j>x_k+y_l. \]
This means that, if we consider the following one-parameter subgroup $\tau$ of $T$:
\[ \begin{array}{rccl}
\tau: & \C^* & \longrightarrow & T\\
 & t & \longmapsto & (\begin{pmatrix}
 t^{x_1} & & \\
  & \ddots & \\
  & & t^{x_m}
 \end{pmatrix},\begin{pmatrix}
 t^{y_1} & & \\
 & \ddots & \\
 & & t^{y_n}
 \end{pmatrix})
\end{array}, \]
it is dominant, regular, and verifies that, for all $\hat{\alpha}\in\hat{\Phi}$,
\begin{equation}\label{eq1}
\hat{v}^{-1}.\hat{\alpha}\in\hat{\Phi}\setminus\hat{\Phi}_{\hat{\mathfrak{p}}}\Longrightarrow\langle\hat{\alpha},\tau\rangle >0.
\end{equation}
Moreover, we can always assume that $\tau$ is $\hat{G}$-regular\footnote{The set of one-parameter subgroups of $T$ verifying condition \eqref{eq1} is an open convex polyhedral cone and, among those subgroups, the not $\hat{G}$-regular ones are elements of some hyperplanes. Thus the set of dominant $\hat{G}$-regular one-parameter subgroups of $T$ verifying condition \eqref{eq1} is not empty.}. As a consequence, $C=\{x_0\}$ is an irreducible component of $Y^\tau$ (cf Remark \ref{remark_dominant_regular_case}).

\vspace{5mm}

This pair $(C,\tau)$ corresponds actually to the same pair as in Theorem \ref{thm_length_zero}: consider the order matrix of the one-parameter subgroup $\tau$ of $T$ and $\hat{w}\in\hat{W}$ the well-defined (since $\tau$ is dominant, regular, $\hat{G}$-regular) Weyl group element we associated to such an order matrix.\\
\underline{First case:} Assume that the coefficients of the matrix $A$ are pairwise distinct. Then the relation between $\hat{w}$ and $\hat{u}$ is simply $\hat{u}=\hat{w}\hat{w}_0$. Then the pair $(C,\tau)$ is exactly the one of Theorem \ref{thm_length_zero} and thus the face $\mathcal{F}(C)$ of $\mathrm{PKron}_{n_1,n_2}$ contains only stable triples. Finally, considering what we have written before, $(\lambda,\mu,\nu)\in\mathcal{F}(C)$.\\
\underline{Second case:} Assume some of the coefficients of $A$ are equal. Then the relation between $\hat{w}$ and $\hat{u}$ is rather that $\hat{u}$ and $\hat{w}\hat{w}_0$ are the same modulo multiplication by $\hat{W}_{\hat{P}}$ on the right. But this means that the two still define the same partial flag in $\hat{G}/\hat{P}$ and, for the same reasons as in the first case, the triple $(\lambda,\mu,\nu)$ is on the face of the Kronecker cone $\mathrm{PKron}_{n_1,n_2}$ given by Theorem \ref{thm_length_zero}. As a consequence it is stable. Moreover, the face of $\mathrm{PKron}_{n_1,n_2}\cap\{(\alpha,\beta,\gamma)\text{ s.t. }\Lb^*_\gamma\text{ is a line bundle on }\hat{G}/\hat{P}\}$ given by $(C,\tau)$ with $C=\{(B/B,\hat{u}\hat{P}/\hat{P})\}\subset Y$ is simply the section of this former face by the subspace $\{(\alpha,\beta,\gamma)\text{ s.t. }\Lb^*_\gamma\text{ is a line bundle on }\hat{G}/\hat{P}\}$.
\end{proof}

The fact that an additive matrix gives in fact a face of minimal dimension of the cone $\mathrm{PKron}_{n_1,n_2}$ was already explained by Manivel in \cite{manivel}.

\subsection{Second case: length 1}\label{cas cardinal 1}

At first we need some more notations: for $i\in\llbracket 1,n_1\rrbracket$, $j\in\llbracket 1,n_2\rrbracket$, and $k\in\llbracket 1,n_1n_2\rrbracket$ (which corresponds to a pair $(i',j')\in\llbracket 1,n_1\rrbracket\times\llbracket 1,n_2\rrbracket$, see Footnote \ref{footnote ordre}), $\varepsilon_i$, $\eta_j$, and $\hat{\varepsilon}_k=\hat{\varepsilon}_{(i',j')}$ are the characters of $T_1$ (the set of diagonal matrices in $G_1$), $T_2$ (same in $G_2$), and $\hat{T}$ respectively, defined by
\[ \varepsilon_i:\begin{pmatrix}
t_1&&\\
&\ddots&\\
&&t_{n_1}
\end{pmatrix}\longmapsto t_i,\]
\[ \eta_j:\begin{pmatrix}
t_1&&\\
&\ddots&\\
&&t_{n_2}
\end{pmatrix}\longmapsto t_j,\]
and
\[ \hat{\varepsilon}_k:\begin{pmatrix}
t_1&&\\
&\ddots&\\
&&t_{n_1n_2}
\end{pmatrix}\longmapsto t_k.\]

\vspace{5mm}

Assume that $\sharp\Phi(v^{-1})=1$, i.e. $\ell(v)=1$. This means that $v=v^{-1}=s_\alpha$, with $\alpha$ a simple root of $G$ (and then $\Phi(s_\alpha)=\{-\alpha\}$). There are two kinds of such $\alpha$'s:
\begin{itemize}
\item the simple roots of $G_1=\gl(V_1)$, which are the $\varepsilon_i-\varepsilon_{i+1}$, for $i\in\llbracket 1,n_1-1\rrbracket$,
\item the simple roots of $G_2=\gl(V_2)$, which are the $\eta_i-\eta_{i+1}$, for $i\in\llbracket 1,n_2-1\rrbracket$.
\end{itemize}
In addition, since we also want $\sharp\hat{\Phi}\big(((\hat{v}\hat{w})^\vee)^{-1}\big)=1$, it is necessary that $((\hat{v}\hat{w})^\vee)^{-1}=s_{\hat{\alpha}}$, for $\hat{\alpha}$ a simple root of $\hat{G}$, i.e. a $\hat{\varepsilon}_k-\hat{\varepsilon}_{k+1}$ for $k\in\llbracket 1,n_1n_2-1\rrbracket$. Then we have
\[ \hat{w}.\hat{\Phi}(s_{\hat{\alpha}})=\{\hat{\varepsilon}_{\hat{w}(k+1)}-\hat{\varepsilon}_{\hat{w}(k)}\}. \]
As a consequence we see that $\alpha$ and $\hat{\alpha}$ will be suitable if and only if
\begin{itemize}
\item $\hat{w}(k)=(i,j)\in\llbracket 1,n_1\rrbracket\times\llbracket 1,n_2\rrbracket\simeq\llbracket 1,n_1n_2\rrbracket$ and $\hat{w}(k+1)=(i,j+1)$,
\item or $\hat{w}(k)=(i,j)$ and $\hat{w}(k+1)=(i+1,j)$.
\end{itemize}
In these cases it is easy to then express $v^{-1}$ and $\hat{v}^{-1}$ and we will get the following result:

\begin{theo}
\begin{itemize}
\item As soon as we have $k\in\llbracket 1,n_1n_2-1\rrbracket$ such that $\hat{w}(k)=(i,j)\in\llbracket 1,n_1\rrbracket\times\llbracket 1,n_2\rrbracket\simeq\llbracket 1,n_1n_2\rrbracket$ and $\hat{w}(k+1)=(i,j+1)$, we have a well-covering pair $(C,\tau)$ (and hence a regular face $\mathcal{F}(C)$ of the Kronecker cone $\mathrm{PKron}_{n_1,n_2}$ containing only stable triples), where
\[ C=\Big\{\Big(\big(\mathsf{1},(j\hphantom{a}j+1)\big).B/B,\hat{w}(k\hphantom{a}k+1)\hat{w}_0.\hat{B}/\hat{B}\Big)\Big\}. \]
\item Likewise, if $k\in\llbracket 1,n_1n_2-1\rrbracket$ is such that $\hat{w}(k)=(i,j)$ and $\hat{w}(k+1)=(i+1,j)$, the pair $(C,\tau)$, where
\[ C=\Big\{\Big(\big((i\hphantom{a}i+1),\mathsf{1}\big).B/B,\hat{w}(k\hphantom{a}k+1)\hat{w}_0.\hat{B}/\hat{B}\Big)\Big\}, \]
is well-covering.
\end{itemize}
\end{theo}

\begin{proof}
We have already seen why these two kinds of properties for $\hat{w}$ give dominant pairs (thanks to Proposition \ref{condition dominant pair}). Then all that remains to be seen is whether these pairs are in fact well-covering, which will be done by looking at the Schubert condition (see Lemma \ref{schubert_condition_lemma}): recall that we have an injective map
\[ \begin{array}{rccl}
\iota: & G/B & \hooklongrightarrow & \hat{G}/\hat{B}(\tau)\\
 & gB & \longmapsto & g\hat{B}(\tau)
\end{array} \]
(with $\hat{B}(\tau)/\hat{B}(\tau)=\hat{w}\hat{B}/\hat{B}$). We can be a little more precise while describing what $\iota$ does on flags:
\[ \begin{array}{rccl}
\iota: & \fl(V_1)\times\fl(V_2) & \hooklongrightarrow & \fl(V_1\otimes V_2)\\
 & \big((E_1\subset\dots\subset E_{n_1-1}),(F_1\subset\dots\subset F_{n_2-1})\big) & \longmapsto & (H_1\subset\dots\subset H_{n_1n_2-1})
\end{array}, \]
with:
\begin{changemargin}{-6mm}{-6mm}
\[ \forall k\in\llbracket 1,n_1n_2-1\rrbracket, \: H_k=H_{k-1}+E_i\otimes F_j, \text{ where }(i,j)=\hat{w}(k)\text{ and }H_0=\{0\},E_{n_1}=V_1,F_{n_2}=V_2. \]
\end{changemargin}
Then we want to look at the intersection between two generic translates of
\[ \hat{X}_{\hat{v}}=\overline{\hat{B}\hat{v}\hat{B}(\tau)/\hat{B}(\tau)} \]
and
\[ \iota(X_v)=\iota(\overline{BvB/B}) \]
(with the usual notation $C=\big\{(v^{-1}B/B,\hat{v}^{-1}\hat{B}/\hat{B})\big\}$). Here, in both cases, there exists $k_0\in\llbracket 1,n_1n_2-1\rrbracket$ such that $\hat{X}_{\hat{v}}$ is of the form $\overline{\hat{B}\hat{w}_0 s_{\hat{\alpha}_{k_0}}\hat{w}^{-1}\hat{B}(\tau)/\hat{B}(\tau)}=\overline{\hat{B}\hat{w}_0 s_{\hat{\alpha}_{k_0}}\hat{B}/\hat{B}}$, i.e. is a Schubert variety of codimension 1, and hence a divisor of $\hat{G}/\hat{B}=\fl(V_1\otimes V_2)$. As a consequence it can be rewritten as
\[ \hat{X}_{\hat{v}}=\big\{(H_1\subset\dots\subset H_{n_1n_2-1})\in\fl(V_1\otimes V_2)\text{ s.t. }\dim\big(H_{k_0}\cap\vect(\hat{e}_1,\dots,\hat{e}_{n_1n_2-k_0}\big)\geq 1\big\}. \]
Then all the translates of $\hat{X}_{\hat{v}}$ correspond to
\[ \big\{(H_1\subset\dots\subset H_{n_1n_2-1})\in\fl(V_1\otimes V_2)\text{ s.t. }\dim(H_{k_0}\cap S)\geq 1\big\}, \]
for all vector subspaces $S$ of $V_1\otimes V_2$ of dimension $n_1n_2-k_0$.

\vspace{5mm}

\noindent\textbf{First case:} $v=\big(\mathsf{1},(j_0\hphantom{a}j_0+1)\big)=s_{\beta_{j_0}}$ (i.e. $\hat{w}(k_0)=(i_0,j_0)$ for some $i_0$ and $\hat{w}(k_0+1)=(i_0,j_0+1)$). Then
\begin{changemargin}{-5mm}{-5mm}
\[ \begin{array}{cl}
 & X_{s_{\beta_{j_0}}}=\overline{Bs_{\beta_{j_0}}B/B}\\
= & \Big\{\Big(\big(\vect(e_1)\subset\dots\subset\vect(e_1,\dots,e_{n_1-1})\big),\big(\vect(f_1)\subset\dots\subset\vect(f_1,\dots,f_{j_0-1})\subset F_{j_0}\\
 & \subset\vect(f_1,\dots,f_{j_0+1})\subset\dots\subset\vect(f_1,\dots,f_{n_2-1})\big)\Big)\in\fl(V_1)\times\fl(V_2)\;;\;F_{j_0}\text{ of dim }j_0\Big\}.
\end{array} \]
\end{changemargin}
As a consequence, $\iota(X_{s_{\beta_{j_0}}})$ is the set of all flags $(H_1\subset\dots\subset H_{n_1n_2-1})$ such that there exists a subspace $F_{j_0}$ of dimension $j_0$ of $V_2$ verifying $\vect(f_1,\dots,f_{j_0-1})\subset F_j\subset\vect(f_1,\dots,f_{j_0+1})$ and, for all $k\in\llbracket 1,n_1n_2\rrbracket$ corresponding to some $(i,j)$,
\[ \left\lbrace\begin{array}{ll}
H_k=H_{k-1}+\vect(e_1,\dots,e_i)\otimes F_{j_0} & \text{when }j=j_0\\
H_k=H_{k-1}+\vect(e_1,\dots,e_i)\otimes\vect(f_1,\dots,f_j) & \text{otherwise}
\end{array}\right.. \]
Then a generic translate of $\hat{X}_{\hat{v}}$ is
\[ \big\{(H_1\subset\dots\subset H_{n_1n_2-1})\in\fl(V_1\otimes V_2)\text{ s.t. }\dim(H_k\cap S)\geq 1\big\} \]
for $S$ of dimension $n_1n_2-k_0$ which does not intersect $\vect(\hat{e}_{\hat{w}(1)},\dots,\hat{e}_{\hat{w}(k_0)})$.
Thus
\begin{changemargin}{-4mm}{-4mm}
\[ \begin{array}{rcl}
\iota(X_{s_{\beta_{j_0}}})\cap Y & = & \big\{(H_1\subset\dots\subset H_{n_1n_2-1})\in\iota(X_{s_{\beta_{j_0}}})\text{ s.t. }H_k\cap S\neq\{0\}\big\}\\
 & = & \big\{(H_1\subset\dots\subset H_{n_1n_2-1})\in\iota(X_{s_{\beta_{j_0}}})\text{ s.t. }F_{j_0}=\vect(f_1,\dots,f_{j_0-1},f_{j_0+1})\big\}
\end{array} \]
\end{changemargin}
(since $\hat{w}(k_0+1)$ is $(i_0,j_0+1)$). This is a singleton, and the Schubert condition is then verified. As a consequence, the pair $(C,\tau)$ is covering, and hence well-covering by Lemma \ref{lemma_covering_well}.

\vspace{5mm}

\noindent\textbf{Second case:} $v=\big((i_0\hphantom{a}i_0+1),\mathsf{1}\big)=s_{\alpha_{i_0}}$. It is sufficient to exchange the roles of $V_1$ and $V_2$.
\end{proof}

\vspace{5mm}

These kinds of properties of $\hat{w}$ are really easy to ``read'' on the order matrix, which allows us to reformulate the previous theorem in the following equivalent way:

\begin{theo}\label{resultat cardinal 1}
For any dominant, regular, $\hat{G}$-regular one-parameter subgroup $\tau$ of $T$, each configuration of the following type in the order matrix of $\tau$:
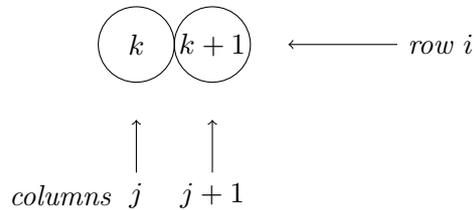
\begin{center}
\begin{tikzpicture}
\node (un) at (0,0) {$k$};
\draw (un) circle (5mm);
\node (deux) at (1,0) {$k+1$};
\draw (deux) circle (5mm);
\node (trois) at (4,0) {row $i$};
\draw[->] (trois) -- (2,0);
\node (quatre) at (0,-2) {$j$};
\draw[->] (quatre) -- (0,-1);
\node (cinq) at (1,-2) {$j+1$};
\draw[->] (cinq) -- (1,-1);
\node at (-1,-2) {columns};
\end{tikzpicture}
\end{center}
gives a well-covering pair $(C,\tau)$, with
\[ C=\Big\{\Big(\big(\mathsf{1},(j\hphantom{a}j+1)\big).B/B,\hat{w}(k\hphantom{a}k+1)\hat{w}_0.\hat{B}/\hat{B}\Big)\Big\}. \]
Hence we obtain a regular face of the Kronecker cone $\mathrm{PKron}_{n_1,n_2}$, of dimension $n_1n_2$ and containing only stable triples:
\[ \mathcal{F}(C)=\big\{(\alpha,\beta,\gamma)\in\mathrm{PKron}_{n_1,n_2}\text{ s.t. }T\text{ acts trivially on the fibre of }\Lb_{\alpha,\beta,\gamma}\text{ over }C\big\}. \]
Likewise, each configuration of the type
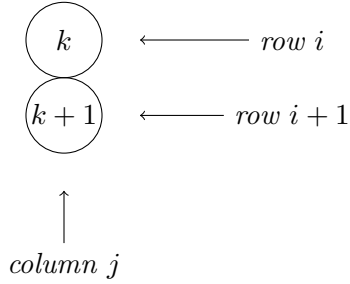
\begin{center}
\begin{tikzpicture}
\node (un) at (0,0) {$k$};
\draw (un) circle (5mm);
\node (deux) at (0,-1) {$k+1$};
\draw (deux) circle (5mm);
\node (trois) at (3,0) {row $i$};
\draw[->] (trois) -- (1,0);
\node (quatre) at (3,-1) {row $i+1$};
\draw[->] (quatre) -- (1,-1);
\node (cinq) at (0,-3) {column $j$};
\draw[->] (cinq) -- (0,-2);
\end{tikzpicture}
\end{center}
gives a well-covering pair $(C,\tau)$ with
\[ C=\Big\{\Big(\big((i\hphantom{a}i+1),\mathsf{1}\big).B/B,\hat{w}(k\hphantom{a}k+1)\hat{w}_0.\hat{B}/\hat{B}\Big)\Big\}. \]
\end{theo}

\paragraph{Example:} The order matrix
\[ \begin{pmatrix}
\text{\ding{192}}&\text{\ding{194}}\\
\text{\ding{193}}&\text{\ding{195}}
\end{pmatrix} \]
which comes for instance from the one-parameter subgroup
\[ \tau:t\longmapsto(\begin{pmatrix}
t&\\
&1
\end{pmatrix},\begin{pmatrix}
t^2&\\
&1
\end{pmatrix}) \]
and corresponds to $\hat{w}=(2\hphantom{a}3)\in\gs_4\simeq\hat{W}$, gives two different well-covering pairs according to the previous theorem:
\begin{itemize}
\item one with $C_1=\Big\{\Big(\big((1\hphantom{a}2),\mathsf{1}\big).B/B,\dps\underset{=(1\hphantom{a}4\hphantom{a}3)}{\underbrace{\hat{w}(1\hphantom{a}2)\hat{w}_0}}.\hat{B}/\hat{B}\Big)\Big\}$,
\item the other with $C_2=\Big\{\Big(\big((1\hphantom{a}2),\mathsf{1}\big).B/B,\dps\underset{=(1\hphantom{a}2\hphantom{a}4)}{\underbrace{\hat{w}(3\hphantom{a}4)\hat{w}_0}}.\hat{B}/\hat{B}\Big)\Big\}$.
\end{itemize}
It is not difficult to see that these cannot be obtained by Theorem \ref{thm_length_zero}: with Lemma \ref{lemma_uniqueness} in mind, we can ``normalise'' these well-covering pairs by action of $G$ so that $C$ is of the form $\{(B/B,\hat{u}\hat{B}/\hat{B})\}$.
\begin{itemize}
\item $\big((1\hphantom{a}2),\mathsf{1}\big).C_1=\Big\{\big(B/B,(1\hphantom{a}2\hphantom{a}4).\hat{B}/\hat{B}\big)\Big\}$,
\item $\big((1\hphantom{a}2),\mathsf{1}\big).C_2=\Big\{\big(B/B,(1\hphantom{a}4\hphantom{a}3).\hat{B}/\hat{B}\big)\Big\}$.
\end{itemize}
Then we see that these $\hat{u}$ cannot be a $\hat{w}\hat{w}_0$, for a $\hat{w}$ coming from an additive matrix, which is what one would get using Theorem \ref{thm_length_zero}. Hence, thanks to Lemma \ref{lemma_uniqueness}, these two examples give two new faces -- $\mathcal{F}(C_1)$ and $\mathcal{F}(C_2)$ -- of the Kronecker cone $\mathrm{PKron}_{2,2}$ which contain only stable triples. The equations of the subspaces spanned by these faces in $\big\{(\alpha,\beta,\gamma)\text{ s.t. }|\alpha|=|\beta|=|\gamma|, \; \ell(\alpha)\leq 2, \; \ell(\beta)\leq 2, \; \ell(\gamma)\leq 4\big\}$ are easy to write:
\begin{itemize}
\item $\left\lbrace\dps\begin{array}{l}
\alpha_1=\gamma_1+\gamma_4\\
\beta_1=\gamma_1+\gamma_2
\end{array}\right.$ for $\mathcal{F}(C_1)$,
\item $\left\lbrace\dps\begin{array}{l}
\alpha_1=\gamma_2+\gamma_3\\
\beta_1=\gamma_1+\gamma_2
\end{array}\right.$ for $\mathcal{F}(C_2)$.
\end{itemize}

\subsection{Third case: length 2}\label{cas cardinal 2}

We follow exactly the same reasoning as in the second case, and keep the same notations. We assume here that $\sharp\Phi(v^{-1})=2$, i.e. $v^{-1}=s_\alpha s_\beta$, for $\alpha$ and $\beta$ distinct simple roots of $G$. If $s_\alpha$ and $s_\beta$ commute, one then has $\Phi(v^{-1})=\{-\alpha,-\beta\}$. If not, $\Phi(v^{-1})=\{-\alpha,-\alpha-\beta\}$. According to the different possibilities there are for $\alpha$ and $\beta$, we then have seven different types for $\Phi(v^{-1})$:
\begin{enumerate}
\item\label{cas 1} $\Phi(v^{-1})=\{\varepsilon_{i+1}-\varepsilon_i,\eta_{j+1}-\eta_j\}$, with $i\in\llbracket 1,n_1-1\rrbracket$, $j\in\llbracket 1,n_2-1\rrbracket$,
\item\label{cas 2} $\Phi(v^{-1})=\{\varepsilon_{i+1}-\varepsilon_i,\varepsilon_{j+1}-\varepsilon_j\}$, with $i,j\in\llbracket 1,n_1-1\rrbracket$ and $|i-j|\geq 2$,
\item\label{cas 3} $\Phi(v^{-1})=\{\eta_{i+1}-\eta_i,\eta_{j+1}-\eta_j\}$, with $i,j\in\llbracket 1,n_1-1\rrbracket$ and $|i-j|\geq 2$,
\item\label{cas 4} $\Phi(v^{-1})=\{\varepsilon_{i+1}-\varepsilon_i,\varepsilon_{i+2}-\varepsilon_i\}$, with $i\in\llbracket 1,n_1-2\rrbracket$,
\item\label{cas 5} $\Phi(v^{-1})=\{\varepsilon_{i+2}-\varepsilon_{i+1},\varepsilon_{i+2}-\varepsilon_i\}$, with $i\in\llbracket 1,n_1-2\rrbracket$,
\item\label{cas 6} $\Phi(v^{-1})=\{\eta_{i+1}-\eta_i,\eta_{i+2}-\eta_i\}$, with $i\in\llbracket 1,n_2-2\rrbracket$,
\item\label{cas 7} $\Phi(v^{-1})=\{\eta_{i+2}-\eta_{i+1},\eta_{i+2}-\eta_i\}$, with $i\in\llbracket 1,n_2-2\rrbracket$.
\end{enumerate}
Then we must also have $((\hat{v}\hat{w})^\vee)^{-1}=s_{\hat{\alpha}}s_{\hat{\beta}}$, for $\hat{\alpha}$, $\hat{\beta}$ simple roots of $\hat{G}$. As before, this yields three kinds of $\hat{\Phi}\big(((\hat{v}\hat{w})^\vee)^{-1}\big)$:
\begin{enumerate}
\item[(a)] $\hat{\Phi}\big(((\hat{v}\hat{w})^\vee)^{-1}\big)=\{\hat{\varepsilon}_{k+1}-\hat{\varepsilon}_k,\hat{\varepsilon}_{k'+1}-\hat{\varepsilon}_{k'}\}$, with $k,k'\in\llbracket 1,n_1n_2-1\rrbracket$ and $|k-k'|\geq 2$,
\item[(b)] $\hat{\Phi}\big(((\hat{v}\hat{w})^\vee)^{-1}\big)=\{\hat{\varepsilon}_{k+1}-\hat{\varepsilon}_k,\hat{\varepsilon}_{k+2}-\hat{\varepsilon}_k\}$, with $k\in\llbracket 1,n_1n_2-2\rrbracket$,
\item[(c)] $\hat{\Phi}\big(((\hat{v}\hat{w})^\vee)^{-1}\big)=\{\hat{\varepsilon}_{k+2}-\hat{\varepsilon}_{k+1},\hat{\varepsilon}_{k+2}-\hat{\varepsilon}_k\}$, with $k\in\llbracket 1,n_1n_2-2\rrbracket$.
\end{enumerate}
And finally some of the cases 1 to 7 are compatible with some of the cases (a) to (c), and will give -- as in Paragraph \ref{cas cardinal 1} -- configurations (concerning $\hat{w}$) providing $v^{-1}$ and $\hat{v}^{-1}$ which verify the condition from Proposition \ref{condition dominant pair}. After removing those which are not possible for a $\hat{w}$ coming from a dominant, regular, $\hat{G}$-regular one-parameter subgroup (i.e. coming from an additive matrix), we obtain the following:

\vspace{5mm}

\noindent\textbf{Configuration \circled{A}} (corresponding to cases \ref{cas 1} and (a)):\\
There exist $k,k'\in\llbracket 1,n_1n_2-1\rrbracket$ such that $|k-k'|\geq 2$ and
\[ \left\lbrace\begin{array}{l}
\hat{w}(k)=(i,j)\\
\hat{w}(k+1)=(i+1,j)\\
\hat{w}(k')=(i',j')\\
\hat{w}(k'+1)=(i',j'+1)
\end{array}\right.. \]
It corresponds to the following situation in the order matrix:
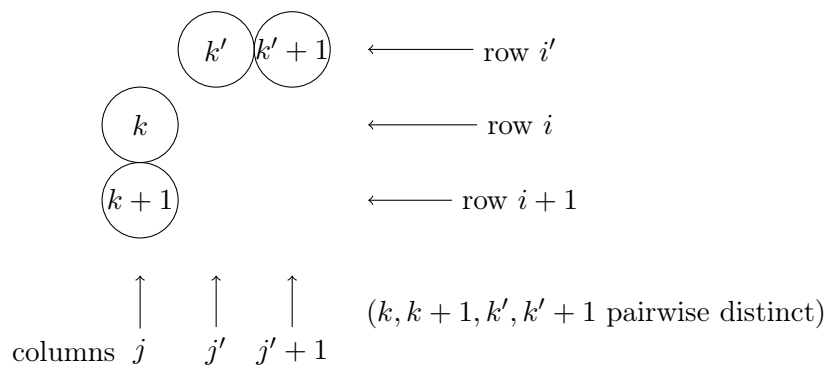
\begin{center}
\begin{tikzpicture}
\node (un) at (0,0) {$k'$};
\draw (un) circle (5mm);
\node (deux) at (1,0) {$k'+1$};
\draw (deux) circle (5mm);
\node (trois) at (4,0) {row $i'$};
\draw[->] (trois) -- (2,0);
\node (quatre) at (0,-4) {$j'$};
\draw[->] (quatre) -- (0,-3);
\node (cinq) at (1,-4) {$j'+1$};
\draw[->] (cinq) -- (1,-3);
\node (six) at (-1,-1) {$k$};
\draw (six) circle (5mm);
\node (sept) at (-1,-2) {$k+1$};
\draw (sept) circle (5mm);
\node (huit) at (4,-1) {row $i$};
\node (neuf) at (4,-2) {row $i+1$};
\draw[->] (huit) -- (2,-1);
\draw[->] (neuf) -- (2,-2);
\node (dix) at (-1,-4) {$j$};
\draw[->] (dix) -- (-1,-3);
\node at (5,-3.5) {($k,k+1,k',k'+1$ pairwise distinct)};
\node at (-2,-4) {columns};
\end{tikzpicture}
\end{center}
This gives:
\[ v^{-1}=\big((i\hphantom{a}i+1),(j'\hphantom{a}j'+1)\big)\quad\text{and}\quad\hat{v}^{-1}=\hat{w}(k\hphantom{a}k+1)(k'\hphantom{a}k'+1)\hat{w}_0 \]
(with $|k-k'|\geq 2$).

\vspace{5mm}

\noindent\textbf{Configuration \circled{B}} (cases \ref{cas 2} and (a)):\\
There exist $k,k'\in\llbracket 1,n_1n_2-1\rrbracket$ such that $|k-k'|\geq 2$ and
\[ \left\lbrace\begin{array}{l}
\hat{w}(k)=(i,j)\\
\hat{w}(k+1)=(i+1,j)\\
\hat{w}(k')=(i',j')\\
\hat{w}(k'+1)=(i'+1,j')
\end{array}\right., \]
with $|i-i'|\geq 2$. It corresponds to the following situation in the order matrix:
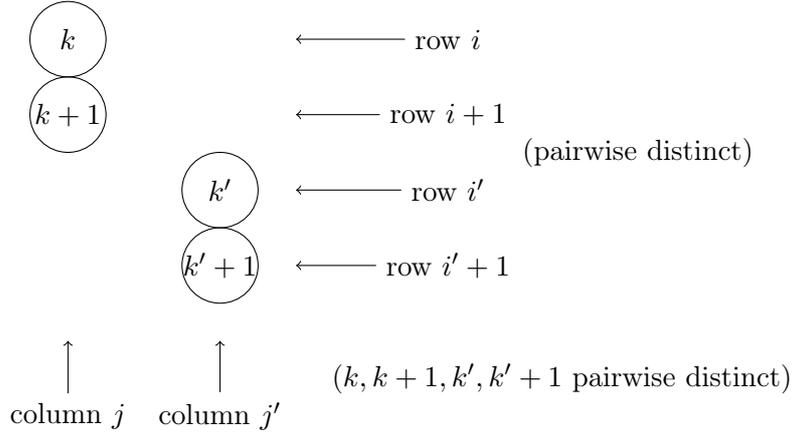
\begin{center}
\begin{tikzpicture}
\node (un) at (-2,2) {$k$};
\draw (un) circle (5mm);
\node (deux) at (-2,1) {$k+1$};
\draw (deux) circle (5mm);
\node (trois) at (3,2) {row $i$};
\draw[->] (trois) -- (1,2);
\node (quatre) at (3,1) {row $i+1$};
\draw[->] (quatre) -- (1,1);
\node (cinq) at (-2,-3) {column $j$};
\draw[->] (cinq) -- (-2,-2);
\node (six) at (0,0) {$k'$};
\draw (six) circle (5mm);
\node (sept) at (0,-1) {$k'+1$};
\draw (sept) circle (5mm);
\node (huit) at (3,0) {row $i'$};
\draw[->] (huit) -- (1,0);
\node (neuf) at (3,-1) {row $i'+1$};
\draw[->] (neuf) -- (1,-1);
\node (dix) at (0,-3) {column $j'$};
\draw[->] (dix) -- (0,-2);
\node at (4.5,-2.5) {($k,k+1,k',k'+1$ pairwise distinct)};
\node at (5.5,0.5) {(pairwise distinct)};
\end{tikzpicture}
\end{center}
This gives
\[ v^{-1}=\big((i\hphantom{a}i+1)(i'\hphantom{a}i'+1),\mathsf{1}\big)\quad\text{and}\quad\hat{v}^{-1}=\hat{w}(k\hphantom{a}k+1)(k'\hphantom{a}k'+1)\hat{w}_0 \]
(with $|i-i'|,|k-k'|\geq 2$).

\vspace{5mm}

\noindent\textbf{Configuration \circled{C}} (cases \ref{cas 1} and (b)):\\
There exists $k\in\llbracket 1,n_1n_2-2\rrbracket$ such that
\[ \left\lbrace\begin{array}{l}
\hat{w}(k)=(i,j)\\
\hat{w}(\{k+1,k+2\})=\{(i+1,j),(i,j+1)\}
\end{array}\right.. \]
It corresponds to two types of situation in the order matrix:
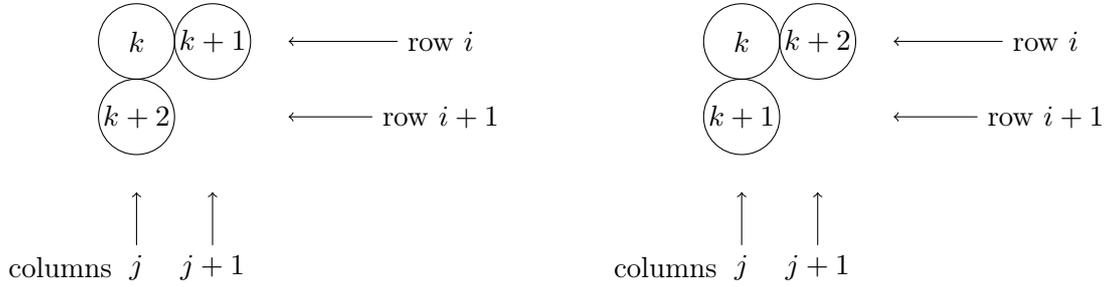
\begin{center}
\begin{tikzpicture}
\node (un) at (0,0) {$k$};
\draw (un) circle (5mm);
\node (deux) at (1,0) {$k+1$};
\draw (deux) circle (5mm);
\node (trois) at (0,-1) {$k+2$};
\draw (trois) circle (5mm);
\node (quatre) at (4,0) {row $i$};
\draw[->] (quatre) -- (2,0);
\node (cinq) at (4,-1) {row $i+1$};
\draw[->] (cinq) -- (2,-1);
\node (six) at (0,-3) {$j$};
\draw[->] (six) -- (0,-2);
\node (sept) at (1,-3) {$j+1$};
\draw[->] (sept) -- (1,-2);
\node at (-1,-3) {columns};
\end{tikzpicture}
\hspace{1cm}
\begin{tikzpicture}
\node (un) at (0,0) {$k$};
\draw (un) circle (5mm);
\node (deux) at (1,0) {$k+2$};
\draw (deux) circle (5mm);
\node (trois) at (0,-1) {$k+1$};
\draw (trois) circle (5mm);
\node (quatre) at (4,0) {row $i$};
\draw[->] (quatre) -- (2,0);
\node (cinq) at (4,-1) {row $i+1$};
\draw[->] (cinq) -- (2,-1);
\node (six) at (0,-3) {$j$};
\draw[->] (six) -- (0,-2);
\node (sept) at (1,-3) {$j+1$};
\draw[->] (sept) -- (1,-2);
\node at (-1,-3) {columns};
\end{tikzpicture}
\end{center}
And this gives
\[ v^{-1}=\big((i\hphantom{a}i+1),(j\hphantom{a}j+1)\big)\quad\text{and}\quad\hat{v}^{-1}=\hat{w}\underset{=(k\hphantom{a}k+1\hphantom{a}k+2)}{\underbrace{(k\hphantom{a}k+1)(k+1\hphantom{a}k+2)}}\hat{w}_0. \]

\vspace{5mm}

\noindent\textbf{Configuration \circled{D}} (cases \ref{cas 1} and (c)):\\
There exists $k\in\llbracket 1,n_1n_2-2\rrbracket$ such that
\[ \left\lbrace\begin{array}{l}
\hat{w}(\{k,k+1\})=\{(i,j+1),(i+1,j)\}\\
\hat{w}(k+2)=(i+1,j+1)
\end{array}\right.. \]
It corresponds to two types of situation in the order matrix:
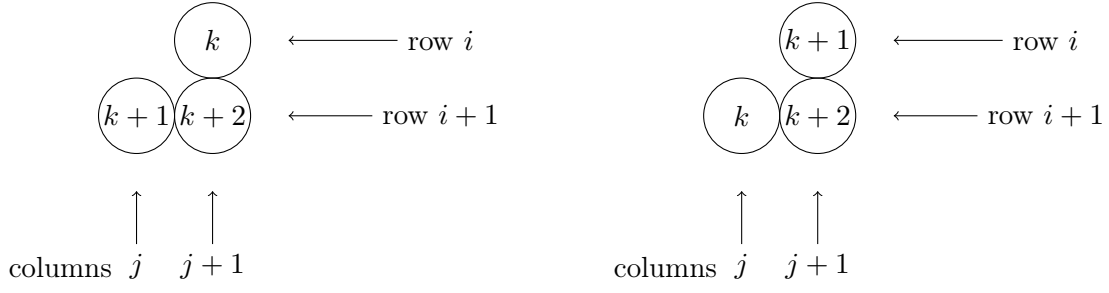
\begin{center}
\begin{tikzpicture}
\node (un) at (1,-1) {$k+2$};
\draw (un) circle (5mm);
\node (deux) at (1,0) {$k$};
\draw (deux) circle (5mm);
\node (trois) at (0,-1) {$k+1$};
\draw (trois) circle (5mm);
\node (quatre) at (4,0) {row $i$};
\draw[->] (quatre) -- (2,0);
\node (cinq) at (4,-1) {row $i+1$};
\draw[->] (cinq) -- (2,-1);
\node (six) at (0,-3) {$j$};
\draw[->] (six) -- (0,-2);
\node (sept) at (1,-3) {$j+1$};
\draw[->] (sept) -- (1,-2);
\node at (-1,-3) {columns};
\end{tikzpicture}
\hspace{1cm}
\begin{tikzpicture}
\node (un) at (1,-1) {$k+2$};
\draw (un) circle (5mm);
\node (deux) at (1,0) {$k+1$};
\draw (deux) circle (5mm);
\node (trois) at (0,-1) {$k$};
\draw (trois) circle (5mm);
\node (quatre) at (4,0) {row $i$};
\draw[->] (quatre) -- (2,0);
\node (cinq) at (4,-1) {row $i+1$};
\draw[->] (cinq) -- (2,-1);
\node (six) at (0,-3) {$j$};
\draw[->] (six) -- (0,-2);
\node (sept) at (1,-3) {$j+1$};
\draw[->] (sept) -- (1,-2);
\node at (-1,-3) {columns};
\end{tikzpicture}
\end{center}
And this gives
\[ v^{-1}=\big((i\hphantom{a}i+1),(j\hphantom{a}j+1)\big)\quad\text{and}\quad\hat{v}^{-1}=\hat{w}\underset{=(k\hphantom{a}k+2\hphantom{a}k+1)}{\underbrace{(k+1\hphantom{a}k+2)(k\hphantom{a}k+1)}}\hat{w}_0. \]

\vspace{5mm}

\noindent\textbf{Configuration \circled{E}} (cases \ref{cas 4} and (b) on the one hand, and \ref{cas 5} and (c) on the other):\\
There exists $k\in\llbracket 1,n_1n_2-2\rrbracket$ such that
\[ \left\lbrace\begin{array}{l}
\hat{w}(k)=(i,j)\\
\hat{w}(k+1)=(i+1,j)\\
\hat{w}(k+2)=(i+2,j)
\end{array}\right.. \]
It corresponds to the following situation in the order matrix:
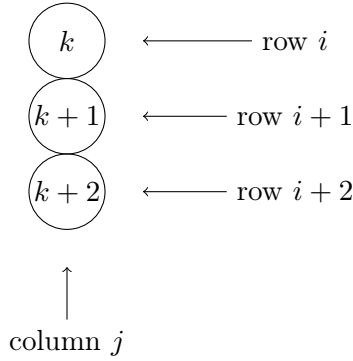
\begin{center}
\begin{tikzpicture}
\node (un) at (0,0) {$k$};
\draw (un) circle (5mm);
\node (deux) at (0,-1) {$k+1$};
\draw (deux) circle (5mm);
\node (trois) at (0,-2) {$k+2$};
\draw (trois) circle (5mm);
\node (quatre) at (3,0) {row $i$};
\draw[->] (quatre) -- (1,0);
\node (cinq) at (3,-1) {row $i+1$};
\draw[->] (cinq) -- (1,-1);
\node (six) at (3,-2) {row $i+2$};
\draw[->] (six) -- (1,-2);
\node (sept) at (0,-4) {column $j$};
\draw[->] (sept) -- (0,-3);
\end{tikzpicture}
\end{center}
And this configuration gives two different pairs $(v^{-1},\hat{v}^{-1})$:
\[ v^{-1}=\big((i\hphantom{a}i+1\hphantom{a}i+2),\mathsf{1}\big)\quad\text{and}\quad\hat{v}^{-1}=\hat{w}(k\hphantom{a}k+1\hphantom{a}k+2)\hat{w}_0, \]
and
\[ v^{-1}=\big((i\hphantom{a}i+2\hphantom{a}i+1),\mathsf{1}\big)\quad\text{and}\quad\hat{v}^{-1}=\hat{w}(k\hphantom{a}k+2\hphantom{a}k+1)\hat{w}_0. \]

\vspace{5mm}

Furthermore the Configurations \circled{B} and \circled{E} each have a ``transposed configuration'' in which the roles of the rows and columns are exchanged. Those two transposed configurations will be denoted respectively by \circled{b} and \circled{e}. They also give a pair $(v^{-1},\hat{v}^{-1})$ (or two, in the case of Configuration \circled{e}) in which the roles of $V_1$ and $V_2$ are exchanged. In other words, $\hat{v}^{-1}$ does not change and -- for instance -- $v^{-1}=\big((i\hphantom{a}i+1)(i'\hphantom{a}i'+1),\mathsf{1}\big)$ (for configuration \circled{B}) becomes $v^{-1}=\big(\mathsf{1},(j\hphantom{a}j+1)(j'\hphantom{a}j'+1)\big)$ (for configuration \circled{b}).

\begin{theo}\label{resultat cardinal 2}
Let $\tau$ be a dominant, regular, $\hat{G}$-regular one-parameter subgroup of $T$. Let
\[ C=\big\{(v^{-1}B/B,\hat{v}^{-1}\hat{B}/\hat{B})\big\}, \]
with $v^{-1}$ and $\hat{v}^{-1}$ coming from one of the Configurations {\normalfont\circled{A}} to {\normalfont\circled{E}} or one of their transposed configurations. Then the pair $(C,\tau)$ is dominant and, as a consequence, gives a face -- not necessarily regular and possibly reduced to zero -- $\mathcal{F}(C)$ of the Kronecker cone $\mathrm{PKron}_{n_1,n_2}$ which contains only almost stable triples.
\end{theo}

\begin{rmk}\label{rmk_2_config}
As we wrote, the two configurations
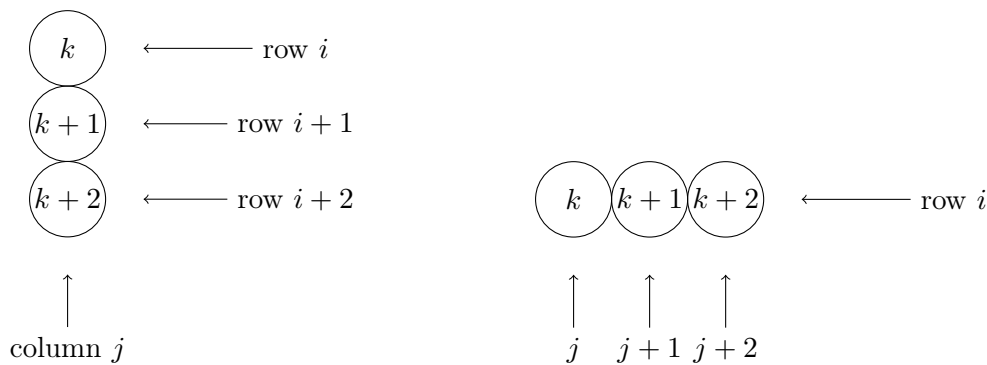
\begin{center}
\begin{tikzpicture}
\node (un) at (0,0) {$k$};
\draw (un) circle (5mm);
\node (deux) at (0,-1) {$k+1$};
\draw (deux) circle (5mm);
\node (trois) at (0,-2) {$k+2$};
\draw (trois) circle (5mm);
\node (quatre) at (3,0) {row $i$};
\draw[->] (quatre) -- (1,0);
\node (cinq) at (3,-1) {row $i+1$};
\draw[->] (cinq) -- (1,-1);
\node (six) at (3,-2) {row $i+2$};
\draw[->] (six) -- (1,-2);
\node (sept) at (0,-4) {column $j$};
\draw[->] (sept) -- (0,-3);
\end{tikzpicture}
\hspace{2cm}
\begin{tikzpicture}
\node (un) at (0,0) {$k$};
\draw (un) circle (5mm);
\node (deux) at (1,0) {$k+1$};
\draw (deux) circle (5mm);
\node (trois) at (2,0) {$k+2$};
\draw (trois) circle (5mm);
\node (quatre) at (5,0) {row $i$};
\draw[->] (quatre) -- (3,0);
\node (cinq) at (0,-2) {$j$};
\draw[->] (cinq) -- (0,-1);
\node (six) at (1,-2) {$j+1$};
\draw[->] (six) -- (1,-1);
\node (sept) at (2,-2) {$j+2$};
\draw[->] (sept) -- (2,-1);
\end{tikzpicture}
\end{center}
(\circled{E} for the former and \circled{e} for the latter) each give two -- a priori different -- dominant pairs . We will see later (cf Paragraph \ref{exemple_3x2}) that such configurations can indeed give two different faces of $\mathrm{PKron}_{n_1,n_2}$.
\end{rmk}

\paragraph{Example:} When one applies Theorem \ref{resultat cardinal 2} to the same order matrix $\begin{pmatrix}
\text{\ding{192}}&\text{\ding{194}}\\
\text{\ding{193}}&\text{\ding{195}}
\end{pmatrix}$ as in Paragraph \ref{cas cardinal 1}, one gets two new dominant pairs: a Configuration \circled{C} and a Configuration \circled{D} can be observed, which give respectively
\[ C_3=\Big\{\Big(\big((1\hphantom{a}2),(1\hphantom{a}2)\big).B/B,\dps\underset{=(1\hphantom{a}4\hphantom{a}3\hphantom{a}2)}{\underbrace{\hat{w}(1\hphantom{a}2\hphantom{a}3)\hat{w}_0}}.\hat{B}/\hat{B}\Big)\Big\} \]
and
\[ C_4=\Big\{\Big(\big((1\hphantom{a}2),(1\hphantom{a}2)\big).B/B,\dps\underset{=(1\hphantom{a}2\hphantom{a}3\hphantom{a}4)}{\underbrace{\hat{w}(2\hphantom{a}4\hphantom{a}3)\hat{w}_0}}.\hat{B}/\hat{B}\Big)\Big\}. \]
Once again we can normalise these $C$'s:
\begin{itemize}
\item $\big((1\hphantom{a}2),(1\hphantom{a}2)\big).C_3=\Big\{\big(B/B,(2\hphantom{a}4).\hat{B}/\hat{B}\big)\Big\}$,
\item $\big((1\hphantom{a}2),(1\hphantom{a}2)\big).C_4=\Big\{\big(B/B,(1\hphantom{a}3).\hat{B}/\hat{B}\big)\Big\}$.
\end{itemize}
The equations defining the subspaces spanned by the corresponding faces -- possibly reduced to zero -- $\mathcal{F}(C_3)$ and $\mathcal{F}(C_4)$ of the Kronecker cone $\mathrm{PKron}_{2,2}$ are respectively:
\[ \left\lbrace\begin{array}{l}
\alpha_1=\gamma_1+\gamma_4\\
\beta_1=\gamma_2+\gamma_4
\end{array}\right.\quad\text{and}\quad\left\lbrace\begin{array}{l}
\alpha_1=\gamma_2+\gamma_3\\
\beta_1=\gamma_2+\gamma_4
\end{array}\right.. \]
One can then check that $\mathcal{F}(C_3)$ and $\mathcal{F}(C_4)$ are indeed not reduced to zero: \linebreak$\big((5,5),(5,5),(3,3,2,2)\big)\in\mathcal{F}(C_3)\cap\mathcal{F}(C_4)$ (it is really a non-stable almost stable triple). But in fact, $\mathcal{F}(C_3)$ and $\mathcal{F}(C_4)$ are equal: the equations of the subspace that they span can be rewritten as
\[ \left\lbrace\begin{array}{l}
\gamma_1=\gamma_2\\
\gamma_3=\gamma_4\\
\alpha_1=\gamma_1+\gamma_3\\
\beta_1=\gamma_1+\gamma_3
\end{array}\right.. \]
So we have actually found only one face of $\mathrm{PKron}_{2,2}$, which is not regular and contains only almost stable triples.

\section{Application to all cases of size 2x2, 3x2, and 3x3}\label{application_small_examples}

\subsection{All order matrices of size 2x2}

For a dominant, regular, $\hat{G}$-regular one-parameter subgroup $\tau$ of $T$, there are only two possible order matrices (i.e. types of additive matrices) in this case:
\[ \begin{pmatrix}
\text{\ding{192}}&\text{\ding{193}}\\
\text{\ding{194}}&\text{\ding{195}}
\end{pmatrix}\qquad\text{and}\qquad\begin{pmatrix}
\text{\ding{192}}&\text{\ding{194}}\\
\text{\ding{193}}&\text{\ding{195}}
\end{pmatrix}, \]
corresponding -- for instance -- respectively to the one-parameter subgroups
\[ \tau_1:t\mapsto(\begin{pmatrix}
t^2&\\
&1
\end{pmatrix},\begin{pmatrix}
t&\\
&1
\end{pmatrix})\qquad\text{and}\qquad\tau_2:t\mapsto(\begin{pmatrix}
t&\\
&1
\end{pmatrix},\begin{pmatrix}
t^2&\\
&1
\end{pmatrix}), \]
which we will from now on denote by $\tau_1=(2,0|1,0)$ and $\tau_2=(1,0|2,0)$. Each one of these order matrices gives one face of $\mathrm{PKron}_{2,2}$, that we will call ``additive'', coming from the result of Manivel and Vallejo (i.e. Theorem \ref{thm_length_zero}). They are respectively given by
\[ C^{(1)}_0=\Big\{\big(B/B,(1\hphantom{a}4)(2\hphantom{a}3)\hat{B}/\hat{B}\big)\Big\} \]
and
\[ C^{(2)}_0=\Big\{\big(B/B,(1\hphantom{a}4)\hat{B}/\hat{B}\big)\Big\}. \]
Then Theorems \ref{resultat cardinal 1} and \ref{resultat cardinal 2} enable us to find other faces (all those coming from the first order matrix will be denoted with an exponent $(1)$ and all others with an exponent $(2)$). While giving them, we will at the same time ``normalise'' the singletons $C$'s as we did in the previous examples: they will all have the form $\big\{(B/B,\hat{u}\hat{B}/\hat{B})\big\}$ with $\hat{u}\in\hat{W}$. (It allows to apply Lemma \ref{lemma_uniqueness} when the pairs are well-covering).

\vspace{5mm}

Theorem \ref{resultat cardinal 1} applied to the first possible order matrix gives two well-covering pairs, with
\[ \begin{array}{l}
C^{(1)}_1=\Big\{\big(B/B,(1\hphantom{a}4\hphantom{a}2\hphantom{a}3)\hat{B}/\hat{B}\big)\Big\},\\
C^{(1)}_2=\Big\{\big(B/B,(1\hphantom{a}3\hphantom{a}2\hphantom{a}4)\hat{B}/\hat{B}\big)\Big\}.
\end{array} \]
Theorem \ref{resultat cardinal 2} gives two dominant ones, with
\[ \begin{array}{l}
C^{(1)}_3=\Big\{\big(B/B,(2\hphantom{a}4\hphantom{a}3)\hat{B}/\hat{B}\big)\Big\},\\
C^{(1)}_4=\Big\{\big(B/B,(1\hphantom{a}2\hphantom{a}3)\hat{B}/\hat{B}\big)\Big\}.
\end{array} \]
Let us do the same for the second possible order matrix (it is actually what we did in the examples of the previous section). With Theorem \ref{resultat cardinal 1}:
\[ \begin{array}{l}
C^{(2)}_1=\Big\{\big(B/B,(1\hphantom{a}2\hphantom{a}4)\hat{B}/\hat{B}\big)\Big\},\\
C^{(2)}_2=\Big\{\big(B/B,(1\hphantom{a}4\hphantom{a}3)\hat{B}/\hat{B}\big)\Big\}.
\end{array} \]
And with Theorem \ref{resultat cardinal 2}:
\[ \begin{array}{l}
C^{(2)}_3=\Big\{\big(B/B,(2\hphantom{a}4)\hat{B}/\hat{B}\big)\Big\},\\
C^{(2)}_4=\Big\{\big(B/B,(1\hphantom{a}3)\hat{B}/\hat{B}\big)\Big\}.
\end{array} \]
These examples being small, it is then not difficult to look in details at every possibly non-regular face (i.e. those coming from Theorem \ref{resultat cardinal 2}). What we find is that these four dominant pairs actually define all the same non-regular and non-zero face of $\mathrm{PKron}_{2,2}$. As before, the subspace of $\big\{(\alpha,\beta,\gamma)\text{ s.t. }|\alpha|=|\beta|=|\gamma|, \; \ell(\alpha)\leq 2, \; \ell(\beta)\leq 2, \; \ell(\gamma)\leq 4\big\}$ spanned by this face has the following equations:
\[ \left\lbrace\begin{array}{l}
\gamma_1=\gamma_2\\
\gamma_3=\gamma_4\\
\alpha_1=\gamma_1+\gamma_3\\
\beta_1=\gamma_1+\gamma_3
\end{array}\right.. \]
In total, we have then found 4 new distinct (by Lemma \ref{lemma_uniqueness}) regular faces of $\mathrm{PKron}_{2,2}$ which contain only stable triples, whereas 2 others were already known. And we have also found 1 non-regular face containing only almost stable triples.

\subsection{All order matrices of size 3x2}\label{exemple_3x2}

Let us do exactly as in the previous case. Five order matrices are possible here:
\[ \begin{pmatrix}
\text{\ding{192}}&\text{\ding{193}}\\
\text{\ding{194}}&\text{\ding{195}}\\
\text{\ding{196}}&\text{\ding{197}}
\end{pmatrix},\qquad\begin{pmatrix}
\text{\ding{192}}&\text{\ding{195}}\\
\text{\ding{193}}&\text{\ding{196}}\\
\text{\ding{194}}&\text{\ding{197}}
\end{pmatrix},\qquad\begin{pmatrix}
\text{\ding{192}}&\text{\ding{193}}\\
\text{\ding{194}}&\text{\ding{196}}\\
\text{\ding{195}}&\text{\ding{197}}
\end{pmatrix},\qquad\begin{pmatrix}
\text{\ding{192}}&\text{\ding{194}}\\
\text{\ding{193}}&\text{\ding{195}}\\
\text{\ding{196}}&\text{\ding{197}}
\end{pmatrix},\qquad\text{and}\qquad\begin{pmatrix}
\text{\ding{192}}&\text{\ding{194}}\\
\text{\ding{193}}&\text{\ding{196}}\\
\text{\ding{195}}&\text{\ding{197}}
\end{pmatrix}. \]
We number them in that order from 1 to 5 and will denote accordingly some possible corresponding one-parameter subgroups:
\vspace{-5mm}
\begin{changemargin}{-1mm}{-1mm}
\[ \tau_1=(4,2,0|1,0),\;\tau_2=(2,1,0|3,0),\;\tau_3=(4,1,0|2,0),\;\tau_4=(4,3,0|2,0),\;\tau_5=(4,2,0|3,0). \]
\end{changemargin}
We then have 5 additive faces with:
\[ \begin{array}{l}
C^{(1)}_0=\Big\{\big(B/B,(1\hphantom{a}6)(2\hphantom{a}5)(3\hphantom{a}4)\hat{B}/\hat{B}\big)\Big\},\\
C^{(2)}_0=\Big\{\big(B/B,(1\hphantom{a}6)(2\hphantom{a}4\hphantom{a}5\hphantom{a}3)\hat{B}/\hat{B}\big)\Big\},\\
C^{(3)}_0=\Big\{\big(B/B,(1\hphantom{a}6)(2\hphantom{a}4\hphantom{a}3\hphantom{a}5)\hat{B}/\hat{B}\big)\Big\},\\
C^{(4)}_0=\Big\{\big(B/B,(1\hphantom{a}6)(2\hphantom{a}5\hphantom{a}3\hphantom{a}4)\hat{B}/\hat{B}\big)\Big\},\\
C^{(5)}_0=\Big\{\big(B/B,(1\hphantom{a}6)(2\hphantom{a}4)(3\hphantom{a}5)\hat{B}/\hat{B}\big)\Big\}.
\end{array} \]
Theorem \ref{resultat cardinal 1} furthermore gives 15 well-covering pairs: since we normalise them by writing $C$ as $\big\{(B/B,\hat{u}\hat{B}/\hat{B})\big\}$, we give in the following table the list of elements $\hat{u}$ obtained, together with the name of the singleton $C$ that they give.
\[ \begin{array}{c|c}
\text{name of }C & \text{element }\hat{u}\text{ giving }C\\
\hline
C^{(1)}_1 & (1\hphantom{a}5\hphantom{a}2\hphantom{a}6)\\
C^{(1)}_2 & (1\hphantom{a}5)(2\hphantom{a}6)(3\hphantom{a}4)\\
C^{(1)}_3 & (1\hphantom{a}6\hphantom{a}2\hphantom{a}5)\\
C^{(2)}_1 & (1\hphantom{a}6)(3\hphantom{a}4\hphantom{a}5)\\
C^{(2)}_2 & (1\hphantom{a}6\hphantom{a}3\hphantom{a}2\hphantom{a}4\hphantom{a}5)\\
C^{(2)}_3 & (1\hphantom{a}4\hphantom{a}5\hphantom{a}3\hphantom{a}2\hphantom{a}6)\\
C^{(2)}_4 & (1\hphantom{a}6)(2\hphantom{a}4\hphantom{a}3)\\
C^{(3)}_1 & (1\hphantom{a}5\hphantom{a}2\hphantom{a}3\hphantom{a}6)\\
C^{(3)}_2 & (1\hphantom{a}4\hphantom{a}3\hphantom{a}5\hphantom{a}2\hphantom{a}6)\\
C^{(3)}_3 & (1\hphantom{a}6)(2\hphantom{a}4\hphantom{a}5)\\
C^{(4)}_1 & (1\hphantom{a}6)(2\hphantom{a}5\hphantom{a}3)\\
C^{(4)}_2 & (1\hphantom{a}6\hphantom{a}3\hphantom{a}4\hphantom{a}2\hphantom{a}5)\\
C^{(4)}_3 & (1\hphantom{a}6\hphantom{a}2\hphantom{a}5\hphantom{a}4)\\
C^{(5)}_1 & (1\hphantom{a}6)(3\hphantom{a}5)\\
C^{(5)}_2 & (1\hphantom{a}6)(2\hphantom{a}4)
\end{array} \]
(The numbers written in exponent between parentheses indicate which order matrix the considered well-covering pairs come from.) Theorem \ref{resultat cardinal 2} gives also 20 dominant pairs, written in the same kind of table:
\[ \begin{array}{c|c}
\text{name of }C & \text{element }\hat{u}\text{ giving }C\\
\hline
C^{(1)}_4 & (1\hphantom{a}5\hphantom{a}2\hphantom{a}6\hphantom{a}3)\\
C^{(1)}_5 & (1\hphantom{a}3\hphantom{a}4\hphantom{a}5)(2\hphantom{a}6)\\
C^{(1)}_6 & (1\hphantom{a}5)(2\hphantom{a}6\hphantom{a}4\hphantom{a}3),\\
C^{(1)}_7 & (1\hphantom{a}4\hphantom{a}6\hphantom{a}2\hphantom{a}5)\\
C^{(2)}_5 & (1\hphantom{a}4\hphantom{a}5\hphantom{a}3\hphantom{a}6)\\
C^{(2)}_6 & (1\hphantom{a}2\hphantom{a}6)(3\hphantom{a}4\hphantom{a}5)\\
C^{(2)}_7 & (1\hphantom{a}6\hphantom{a}5)(2\hphantom{a}4\hphantom{a}3)\\
C^{(2)}_8 & (1\hphantom{a}6\hphantom{a}3\hphantom{a}2\hphantom{a}4)\\
C^{(3)}_4 & (1\hphantom{a}5)(2\hphantom{a}4\hphantom{a}6)\\
C^{(3)}_5 & (1\hphantom{a}5\hphantom{a}2)(3\hphantom{a}6)\\
C^{(3)}_6 & (1\hphantom{a}4\hphantom{a}6\hphantom{a}2\hphantom{a}3\hphantom{a}5)\\
C^{(3)}_7 & (1\hphantom{a}3\hphantom{a}6)(2\hphantom{a}5)\\
C^{(3)}_8 & (1\hphantom{a}5\hphantom{a}2\hphantom{a}3\hphantom{a}4\hphantom{a}6)\\
C^{(4)}_4 & (1\hphantom{a}6\hphantom{a}2\hphantom{a}5\hphantom{a}4\hphantom{a}3)\\
C^{(4)}_5 & (1\hphantom{a}6\hphantom{a}4)(2\hphantom{a}5)\\
C^{(4)}_6 & (1\hphantom{a}5\hphantom{a}3)(2\hphantom{a}6)\\
C^{(4)}_7 & (1\hphantom{a}5)(2\hphantom{a}6\hphantom{a}4)\\
C^{(4)}_8 & (1\hphantom{a}4)(2\hphantom{a}5\hphantom{a}6)\\
C^{(5)}_3 & (1\hphantom{a}5\hphantom{a}6\hphantom{a}2\hphantom{a}4)\\
C^{(5)}_4 & (1\hphantom{a}5\hphantom{a}3\hphantom{a}6\hphantom{a}2)
\end{array} \]
If we check here one by one whether these dominant pairs are actually well-covering, we find that eight of them indeed are: those given by $C_5^{(2)}$, $C_6^{(2)}$, $C_7^{(2)}$, $C_8^{(2)}$, $C_7^{(3)}$, $C_8^{(3)}$, $C_4^{(4)}$, and $C_5^{(4)}$. Let us give two examples, for $C_5^{(2)}$ and $C_6^{(2)}$: the equations defining the subspaces of $\big\{(\alpha,\beta,\gamma)\text{ s.t. }|\alpha|=|\beta|=|\gamma|, \; \ell(\alpha)\leq 3, \; \ell(\beta)\leq 2, \; \ell(\gamma)\leq 6\big\}$ spanned respectively by $\mathcal{F}(C^{(2)}_5)$ and $\mathcal{F}(C^{(2)}_6)$ are respectively
\[ \left\lbrace\begin{array}{l}
\alpha_1=\gamma_1+\gamma_5\\
\alpha_2=\gamma_2+\gamma_6\\
\beta_1=\gamma_1+\gamma_2+\gamma_3
\end{array}\right.\qquad\text{and}\qquad\left\lbrace\begin{array}{l}
\alpha_1=\gamma_1+\gamma_6\\
\alpha_2=\gamma_2+\gamma_4\\
\beta_1=\gamma_1+\gamma_2+\gamma_3
\end{array}\right.. \]
We can for instance notice that $\big((4,3,2),(8,1),(4,3,1,1)\big)$ belongs to $\mathcal{F}(C^{(2)}_5)$, whereas $\big((4,3,1),(7,1),(4,2,1,1)\big)$ is in $\mathcal{F}(C^{(2)}_6)$. Then, since these faces contain each some triple $(\alpha,\beta,\gamma)$ of partitions which is such that either $\alpha$ and $\beta$ are regular (meaning that the parts of $\alpha$ are pairwise distinct, as are those of $\beta$), or $\gamma$ is regular (likewise), \cite[Theorem 12]{ressayre} assures that these two dominant pairs are in fact well-covering. Moreover, Lemma \ref{lemma_uniqueness} proves that the two (thus regular) faces are distinct. This is interesting because they come from the same Configuration \circled{E}, appearing in the first column of the order matrix number 2. Hence this kind of configuration can indeed give two different regular faces of $\mathrm{PKron}_{n_1,n_2}$ (cf Remark \ref{rmk_2_config}).

\vspace{5mm}

Lemma \ref{lemma_uniqueness} furthermore assures that all the regular faces corresponding to the 28 well-covering pairs that we presented are pairwise distinct. Looking in more details at the 12 other dominant pairs, which are not well-covering, we can see that they in fact give only two distinct non-regular faces of $\mathrm{PKron}_{3,2}$ containing almost stable triples: the equations of the subspaces that they respectively span are
\[ \left\lbrace\begin{array}{l}
\gamma_1=\gamma_2\\
\gamma_3=\gamma_4\\
\gamma_5=\gamma_6\\
\alpha_1=\alpha_2=\gamma_1+\gamma_3\\
\beta_1=\gamma_1+\gamma_3+\gamma_5
\end{array}\right.\qquad\text{and}\qquad\left\lbrace\begin{array}{l}
\gamma_1=\gamma_2\\
\gamma_3=\gamma_4\\
\gamma_5=\gamma_6\\
\alpha_1=2\gamma_1\\
\alpha_2=\gamma_3+\gamma_5\\
\beta_1=\gamma_1+\gamma_3+\gamma_5
\end{array}\right.. \]
In total we then obtained 23 new regular faces of $\mathrm{PKron}_{3,2}$ which contain only stable triples, whereas 5 others were already known. We also got 2 other non-regular faces, containing only almost stable triples.

\begin{rmk}
While doing the previous computations, we noticed that Configurations \circled{A} and \circled{E} ultimately gave a well-covering pair each time, whereas Configurations \circled{C} and \circled{D} yielded only dominant pairs which were not well-covering (our examples were too small to observe a Configuration \circled{B}). One could wonder whether this is always the case.
\end{rmk}

\subsection{All order matrices of size 3x3}

For this case the numbers begin to become much larger: there are 36 possible order matrices (i.e. 36 types of additive matrices) of size 3$\times$3. As a consequence we do not write all of them here\footnote{For the interested reader, they can be found online at \url{http://math.univ-lyon1.fr/~pelletier/recherche/matrices_additives_3x3.pdf}, along with the number of well-covering and dominant pairs that each provides.}. First Manivel and Vallejo's theorem yields 36 additive faces of $\mathrm{PKron}_{3,3}$ with this 36 additive matrices. But then if we look in details at these matrices, we find that Theorem \ref{resultat cardinal 1} gives 144 well-covering pairs, i.e. 144 regular faces containing only stable triples. Moreover, Theorem \ref{resultat cardinal 2} adds 232 dominant pairs to this, i.e. 232 faces of $\mathrm{PKron}_{3,3}$ -- possibly non-regular and not necessarily pairwise distinct -- containing only almost stable triples. Considering what happened in the two previous cases, we can hope that some of those dominant pairs are in fact well-covering. It would be of course possible to check whether this is true, but it is far too tedious to do it here.

\vspace{5mm}

Let us nevertheless give one detailed example of a new face of $\mathrm{PKron}_{3,3}$ that we can obtain with our results: look at the order matrix
\[ \begin{pmatrix}
\text{\ding{192}}&\text{\ding{193}}&\text{\ding{198}}\\
\text{\ding{194}}&\text{\ding{196}}&\text{\ding{199}}\\
\text{\ding{195}}&\text{\ding{197}}&\text{\ding{200}}
\end{pmatrix} \]
coming for instance from the dominant, regular, $\hat{G}$-regular one-parameter subgroup $\tau=(4,1,0|7,5,0)$ of $T$. Theorem \ref{resultat cardinal 2} tells us that the pair with
\[ C=\Big\{\big(B/B,(1\hphantom{a}9\hphantom{a}4\hphantom{a}2\hphantom{a}6)(5\hphantom{a}8)\hat{B}/\hat{B}\Big)\Big\} \]
is dominant (it comes from a Configuration \circled{E} in the third column of the matrix). And we can compute the equations of the subspace spanned by $\mathcal{F}(C)$ in $\big\{(\alpha,\beta,\gamma)\text{ s.t. }|\alpha|=|\beta|=|\gamma|, \; \ell(\alpha)\leq 3, \; \ell(\beta)\leq 3, \; \ell(\gamma)\leq 9\big\}$:
\[ \left\lbrace\begin{array}{l}
\alpha_1=\gamma_4+\gamma_6+\gamma_7\\
\alpha_2=\gamma_1+\gamma_2+\gamma_8\\
\beta_1=\gamma_1+\gamma_3+\gamma_4\\
\beta_2=\gamma_2+\gamma_5+\gamma_6
\end{array}\right.. \]
Then one can notice that $\big((6,5,4),(7,6,2),(3,2^6)\big)\in\mathcal{F}(C)$. As a consequence \cite[Theorem 12]{ressayre} assures that the pair $(C,\tau)$ is well-covering, and $\mathcal{F}(C)$ is indeed a regular face of $\mathrm{PKron}_{3,3}$ containing only stable triples.

\bibliographystyle{alpha}
\bibliography{Biblio_these}
\end{document}